\documentclass[12pt]{amsart}
\usepackage{amssymb,latexsym,amsmath,amsthm,mathrsfs,yfonts}
\usepackage{fullpage,enumerate}
\input xy
\xyoption{all}
\def\ra{\rightarrow}
\newcommand{\beq}{\begin{eqnarray}}
\newcommand{\eeq}{\end{eqnarray}}

\newcommand\CC{\mathbb C}
\newcommand\ZZ{\mathbb Z}
\newcommand\TT{\mathbb T}

\newcommand{\cL}{\mathcal{L}}
\newcommand{\cG}{\mathcal{G}}
\newcommand{\cI}{\mathcal{I}}

\def\cA{{\mathcal A}}

\def\cB{{\mathcal B}}

\def\cF{{\mathcal F}}
\def\cG{{\mathcal G}}

\def\cI{{\mathcal I}}

\def\cL{{\mathcal L}}
\def\cM{{\mathcal M}}

\newtheorem{theorem}{Theorem}[section]
\newtheorem{lemma}[theorem]{Lemma}

\theoremstyle{definition}

\theoremstyle{remark}
\newtheorem{remark}[theorem]{Remark}

\begin{document}

\title{T-duality and the exotic chiral de Rham complex}

\author{Andrew Linshaw}
\address[A Linshaw]{Department of Mathematics,
University of Denver, Denver, CO 80208, USA}
\email{andrew.linshaw@du.edu}

\author{Varghese Mathai}
\address[V Mathai]{Department of Pure Mathematics, University of Adelaide,
Adelaide, SA 5005, Australia}
\email{mathai.varghese@adelaide.edu.au}

\begin{abstract}
Let $Z$ be a principal circle bundle over a base manifold $M$ equipped with an integral closed $3$-form $H$ called the flux. Let $\widehat{Z}$ be the T-dual circle bundle over $M$ with flux $\widehat{H}$. Han and Mathai recently constructed the $\mathbb{Z}_2$-graded space of exotic differential forms $\mathcal{A}^{\bar{k}}(\widehat{Z})$. It has an additional $\mathbb{Z}$-grading such that the degree zero component coincides with the space of invariant twisted differential forms $\Omega^{\bar{k}}(\widehat{Z}, \widehat{H})^{\widehat{\TT}}$, and it admits a differential that extends the twisted differential $d_{\widehat{H}} = d + \widehat{H}$. The T-duality isomorphism $\Omega^{\bar{k}}(Z,H)^{\TT} \rightarrow \Omega^{\overline{k+1}}(\widehat{Z}, \widehat{H})^{\widehat{\TT}}$ of Bouwknegt, Evslin and Mathai extends to an isomorphism $\Omega^{\bar{k}}(Z,H) \rightarrow \mathcal{A}^{\overline{k+1}}(\widehat{Z})$. In this paper, we introduce the exotic chiral de Rham complex $\mathcal{A}^{\text{ch},\widehat{H},\bar{k}}(\widehat{Z})$ which contains $\mathcal{A}^{\bar{k}}(\widehat{Z})$ as the weight zero subcomplex. We give an isomorphism $\Omega^{\text{ch},H,\bar{k}}(Z) \rightarrow \mathcal{A}^{\text{ch},\widehat{H},\overline{k+1}}(\widehat{Z})$ where $\Omega^{\text{ch},H,\bar{k}}(Z)$ denotes the twisted chiral de Rham complex of $Z$, which chiralizes the above T-duality map. \end{abstract}


\maketitle

\section{Introduction}

A space of exotic differential forms with an equivariantly flat superconnection \cite{MQ} was first defined on loop space in the paper \cite{HM15}, which  we now briefly recall here.   
  
Let $(H, B_\alpha, F_{\alpha\beta}, L_{\alpha\beta})$ denote a gerbe with connection on $Z$ (cf. \cite{Brylinski}), where 
$(H, B_\alpha, F_{\alpha\beta})$ denotes the Deligne class of the closed integral 3-form $H$ with respect to 
a Brylinski open cover (cf. \cite{HM15}), and $ L_{\alpha\beta}$
denotes the line bundles on double overlaps that determines the gerbe $\cG$. The holonomy of the gerbe is then a line 
bundle $\cL$ with connection $d+\tau(B_\alpha)$ having curvature $\tau(H)$ on loop space $LZ$, where $\tau$ denotes the transgression map.
We consider the space of invariant exotic differential forms on loop space $LM$, $\Omega^{\bar k}(LZ, \cL)^{S^1} $, with exotic differential
$D=\nabla^\cL -i_K + \bar H$, where $\nabla^\cL$ is the connection on the holonomy line bundle $\cL$ given locally by
$d+\tau(B_\alpha)$, $i_K$ is contraction by the rotation vector field $K$, and $ \bar H$ denotes the 3-form on $LZ$ given by 
the canonical extension of $H$ to loop space. Then a computation in \cite{HM15} shows that $D^2 = L_K$, so that 
$D^2 = 0$ on $\Omega^{\bar k}(LZ, \cL)^{S^1} $.

Let $\TT \to Z \stackrel{\pi}{\to} M$ be a principal circle bundle over a base manifold $M$ with background $\TT$-invariant flux $H$, which is a closed 3-form on $Z$. Then there is a T-dual circle bundle $\widehat{\TT}\to\widehat{Z} \stackrel{\hat\pi}{\to} X$ with T-dual background $\widehat{\TT}$-invariant flux $\widehat{H}$ which is a closed 3-form on $\widehat{Z}$, such that $c_1(Z) = \widehat{\pi}_*[\widehat{H}]$ and $c_1(\widehat{Z}) = \pi_*[ H]$, and the constraint that $[H]=[\widehat{H}]$ on the correspondence space
$Z\times_M \widehat{Z}$  ensures that $[\widehat{H}]$ is uniquely defined. This is the setting of \cite{BEM,BEM2}. Let $L = Z\times_\TT\CC$ and $\widehat L = \widehat{Z} \times_{\widehat{\TT}} \CC$ denote the associated line bundles over the base space $M$.
 
 The precise relation between \cite{HM15} and \cite{HM18} is that when $Z$ is the total space of a principal circle bundle, then
  there is a natural infinite sequence of embeddings $\iota_n : Z \to LZ$ defined by $\iota_n(x): S^1\ni t\mapsto  \gamma_x(t)=t^n\cdot x$, for all $n\in \ZZ$. We consider such sequence of embeddings motivated by the fact that there are $\mathbb{Z}$ many connected components 
  in the loop space $L\TT$. We have $\iota_n^*(\cL) \cong  \pi^*(\widehat L)^{\otimes n}$ since they have the same Chern class. The loop space $LZ$ has the natural circle action by rotating loops, and $Z$ has a circle action as the total space of a circle bundle. To tell the difference of these two circle actions, we use $S^1$ for the  circle action by rotating loops, and $\TT$ for the free circle action on $Z$ as a principal circle bundle. We have that for $n\neq 0$, 
$$\iota_n^*: \Omega^{\bar k}(LZ, \cL)^{S^1} \longrightarrow \Omega^{\bar k}(Z, \pi^*(\widehat{ L}^{\otimes n}))^{\TT}$$
 intertwines the equivariantly flat superconnections $D$ and $\pi^*\nabla^{\widehat L^{\otimes n}}-\iota_{n{v}}+{H}$ on both spaces. Here $v$ is the vector field on $Z$ which infinitesimally generates the action of $\TT$. This point of view does motivate us to develop the exotic theories on $Z$.

Recall that the local T-duality rules, called the Buscher rules, were written in \cite{Bus, AABL}. The relation of T-duality with K-theory in the absence of an H-flux was studied in \cite{Hori1,Hori}, and in the presence of an H-flux in \cite{BEM,BEM2} where for the first time there was topology change between spacetime and its T-dual. See also \cite{BS} and \cite{MR05,MR05b,MR06,BHM} for alternate approaches to T-duality.
In \cite{HM18}, the T-duality isomorphism given in \cite{BEM,BEM2} was extended to a mapping from the full space of complex-valued differential forms defined on a principal circle bundle. In doing so, the striking result obtained was that the T-dual data of this space is given by the space of exotic differential forms defined on the T-dual principal circle bundle. The definition of exotic differential forms was inspired by their previous work, \cite{HM15}.

In order to define this T-duality mapping, let $L, \widehat L$ denote the complex line bundles associated to the circle bundles $Z, \widehat{Z}$ with the standard representation of the circle on the complex plane respectively. The {\em exotic differential forms} are then given by
$$\mathcal{A}^{\bar k}(Z)=\bigoplus_{n\in \ZZ}\mathcal{A}^{\bar k}_n(Z)^{\TT}:=\bigoplus_{n\in \ZZ}\Omega^{\bar k}(Z, \pi^*(\widehat{L}^{\otimes n}))^{\TT},$$
$$\mathcal{A}^{\bar k}(\widehat{Z})=\bigoplus_{n\in \ZZ}\mathcal{A}^{\bar k}_n(\widehat{Z})^{\widehat \TT}:=\bigoplus_{n\in \ZZ}\Omega^{\bar k}(\widehat{Z}, \widehat{\pi}^*(L^{\otimes n}))^{\widehat \TT}$$
for $\bar k= k \mod 2$, and where we have taken the direct sum above to be the Fr\'{e}chet space completion of the standard direct sum. This definition of the direct sum (as a completion) will be the implicit definition from here on out when using direct sums in the context of the exotic structures. Note that in \cite{HM18}, the notation $\mathcal{A}^{\bar k}(Z)^{\TT}$ and $\mathcal{A}^{\bar k}(\widehat{Z})^{\widehat{\TT}}$ is used instead of $\mathcal{A}^{\bar k}(Z)$ and $\mathcal{A}^{\bar k}(\widehat{Z})$; we have dropped the $\TT$ and $\widehat{\TT}$ invariant notation on the direct sums for simplicity.

Now define the subspace of weight $-n$ differential forms on $Z$ to be,
\begin{align}
 \Omega^*_{-n}(Z):=\{\omega \in \Omega^*(Z)|\ {\rm Lie}_{v}\omega=-n\omega\},
\end{align}
where ${\rm Lie}_{v}$ denote the Lie derivative along $v$. We observe that
$$\Omega^{\bar k}_0(Z)=\Omega^{\bar k}(Z)^{\TT}, \ \  \mathcal{A}^{\overline{k+1}}_0(\widehat{Z})^{\widehat \TT}=\Omega^{\overline{k+1}}(\widehat{Z})^{\widehat \TT}.$$
Then under the above choices of Riemannian metrics and flux forms, the results of \cite{HM18} show that there is a sequence of  isometries,
\begin{align}
\tau_n \colon \Omega^{\bar k}_{-n}(Z) \to \mathcal{A}^{\overline{k+1}}_n(\widehat{Z})^{\widehat \TT},
\end{align}
defined by the {\em exotic Hori formula} from $Z$ to $\widehat{Z}$ given in \cite{HM18} for $\bar k= k \mod 2$, where the twisted de Rham differential $d+H$ maps to the differential $-(\widehat{\pi}^*\nabla^{L^{\otimes n}}-\iota_{n\widehat{v}}+\widehat{H})$, and we observe that $\tau_0=T$.  One similarly has 
a sequence of  isometries,
\begin{equation}
\sigma_n \colon\mathcal{A}^{\bar k}_n(Z)^{\TT} \to \Omega^{\overline{k +1}}_{-n}(\widehat{Z}),
\end{equation}
defined by the {\em inverse exotic Hori formula} form $Z$ to $\widehat{Z}$ given in equation \cite{HM18} for $\bar k= k \mod 2$, where the differential $\pi^*\nabla^{\widehat L^{\otimes n}}-\iota_{n{v}}+{H}$
maps to  the twisted de Rham differential $-(d+\widehat{H})$, and $\sigma_0=T$. Similarly, one can define the sequences of isometries $\widehat\tau_n, \widehat\sigma_n$ on $\widehat{Z}$. Although the extension of the Fourier-Mukai transform 
to all differential forms on $Z$ is slightly asymmetric, one has the following crucial identities, verified in \cite{HM18}:
\begin{align}
{\rm -Id} =\widehat\sigma_n \circ \tau_n \colon \Omega^{\bar k}_{-n}(Z)  \longrightarrow \Omega^{\bar k}_{-n}(Z),\\
{\rm -Id} =\widehat\tau_n \circ \sigma_n  \colon \mathcal{A}^{\bar k}_n(Z)^{\TT} \longrightarrow \mathcal{A}^{\bar k}_n(Z)^{\TT}.  
\end{align}
This is interpreted as saying that T-duality, when applied twice, returns the object to minus itself, which arises due to the convention of integration along the fiber. This was a result previously verified in \cite{BEM,BEM2} for the special case of when $n=0$.

This shows that for each of either $Z$ or $\widehat{Z}$, there are two theories (at degree 0 the two theories coincide), and there are also graded isomorphisms between the two theories of both sides. 

Moreover, when $n\neq 0$ 
the complex $(\mathcal{A}^{\bar k+1}_n(\widehat{Z})^{\widehat \TT}, \widehat{\pi}^*\nabla^{L^{\otimes n}}-\iota_{n\widehat{v}}+\widehat{H})$ has vanishing cohomology. Therefore, when $n\neq 0$ the complex $(\Omega^{\bar k}_{-n}(Z), d+H)$ also has 
vanishing cohomology. In \cite{HM18},  an explicit homotopy is constructed to show this. We mention that, inspired by  \cite{HM18}, exotic Courant algebroids were defined in \cite{Crilly} where the T-duality isomorphism in \cite{cavalcanti} for invariant Courant algebroids was extended to a T-duality isomorphism of exotic Courant algebroids.

\subsection{Chiralization} The chiral de Rham complex is a sheaf of vertex algebra $\Omega^{\text{ch}}_M$ on any smooth manifold $M$ that was introduced by Malikov, Schechtman, and Vaintrob in \cite{MSV1}. It has had a tremendous impact on string theory in the last 20 years; see for example \cite{Borisov,KO,Witten07}. The global section algebra $\Omega^{\text{ch}}(M)$ has an $\mathbb{N}$-grading by conformal weight, and it chiralizes the de Rham complex $(\Omega(M), d)$ in the sense that it admits a differential $D$ which preserves the weight spaces, and the weight zero subcomplex $(\Omega^{\text{ch}}(M)[0], D)$ is isomorphic to $(\Omega(M), d)$.

In this paper, we chiralize the space of exotic differential forms on $\widehat{Z}$ to the vertex algebra of {\it exotic chiral differential forms} 
$$\mathcal{A}^{{\rm ch}, \widehat{H}}(\widehat{Z}) = \bigoplus_{n \in \mathbb{Z}}\mathcal{A}_n^{{\rm ch}, \widehat{H}}(\widehat{Z}).$$ This is the global section algebra of a sheaf of $\mathbb{Z}_2$-graded vertex algebras on $\widehat{Z}$. We denote the graded components by $\mathcal{A}_n^{{\rm ch}, \widehat{H},\bar{k}}(\widehat{Z})$ for $\bar{k} \in \mathbb{Z}_2$. Unlike the chiral de Rham complex, $\mathcal{A}^{{\rm ch}, \widehat{H}}(\widehat{Z})$ is naturally equipped only with a filtration, not a grading, by weight: 
\begin{equation} \label{intro:filtration} \mathcal{A}^{{\rm ch}, \widehat{H}}(\widehat{Z})_{[0]} \subseteq \mathcal{A}^{{\rm ch}, \widehat{H}}(\widehat{Z})_{[1]} \subseteq \mathcal{A}^{{\rm ch}, \widehat{H}}(\widehat{Z})_{[2]} \subseteq \cdots,\qquad \mathcal{A}^{{\rm ch}, \widehat{H}}(\widehat{Z})= \bigcup_{i\geq 0} \mathcal{A}^{{\rm ch}, \widehat{H}}(\widehat{Z})_{[i]}.\end{equation} We also equip $\mathcal{A}^{{\rm ch}, \widehat{H}}(\widehat{Z})$ with an exotic differential $D_{\widehat{Z},\widehat{H}}$ which shifts the $\mathbb{Z}_2$-grading and preserves the weight filtration. This structure chiralizes the exotic differential forms in the sense that for all $n$, the weight zero subcomplex $(\mathcal{A}_n^{{\rm ch}, \widehat{H},\bar k}(\widehat{Z})_{[0]}, D_{\widehat{Z},\widehat{H}})$ is isomorphic to $(\cA_n^{\bar{k}}(\widehat{Z}), \pi^*\nabla^{L^{\otimes n}}-\iota_{n{\widehat{v}}}+\widehat{H})$. In fact, $D_{\widehat{Z},\widehat{H}}$ is a square-zero operator on $\mathcal{A}^{{\rm ch}, \widehat{H}}(\widehat{Z})$; the proof requires a very delicate calculation and depends crucially on the nonassociativity of the normally ordered product. 

Using the flux $H$ on $Z$, there is an $H$-twisted version of the chiral de Rham complex $\Omega^{\text{ch},H}(Z)$ which was introduced in \cite{LM15}. It turns out to be isomorphic to $\Omega^{\text{ch}}(Z)$ via an {\it untwisting trick}; see Theorem 3 of \cite{LM15}, and for convenience, we use $\Omega^{\text{ch},H}(Z)$ instead of $\Omega^{\text{ch}}(Z)$ throughout this paper. As in the case of differential forms, the $\TT$-action on $Z$ induces a Fourier decomposition $\Omega^{\text{ch},H}(Z) = \bigoplus_{n\in \mathbb{Z}} \Omega^{\text{ch},H}_n(Z)$, which again denotes the Fr\'{e}chet space completion of the standard direct sum. There is also a $\mathbb{Z}_2$-grading, and we denote the graded components by $\Omega^{\text{ch},H,\bar{k}}_n(Z)$ for $\bar{k} \in \mathbb{Z}_2$. Our main result is that T-duality gives a degree shifting linear isomorphism
\begin{equation} \label{intro:main}
\tau^{\text{ch}}_n: \Omega^{{\rm ch}, H, \bar{k}}_{-n}(Z) \rightarrow \mathcal{A}_{n}^{\text{ch}, \widehat{H},\overline{k+1}}(\widehat{Z}),
\end{equation}
for all $n\in \mathbb{Z}$. This map preserves the weight filtration and coincides with $\tau_n: \Omega^{\bar{k}}_{-n}(Z) \rightarrow \mathcal{A}^{\overline{k+1}}_n(\widehat{Z})^{\widehat \TT}$ on the weight zero subspace. These isomorphisms combine to yield a linear isomorphism 
\begin{equation} \label{intro:main2} \tau^{\text{ch}}: \Omega^{{\rm ch}, H,\bar{k}}(Z) \rightarrow \mathcal{A}^{\text{ch}, \widehat{H},\overline{k+1}}(\widehat{Z}).\end{equation}
In fact, $\tau^{\text{ch}}$ is more than a linear isomorphism. We will also define a vertex algebra isomorphism $\phi^{\text{ch}}: \Omega^{\text{ch},H}(Z) \rightarrow  \mathcal{A}^{\text{ch},\widehat{H}}(\widehat{Z})$ which preserves the $\mathbb{Z}_2$-grading. Regarding $\Omega^{\text{ch},H}(Z)$ and $\mathcal{A}^{\text{ch}, \widehat{H}}(\widehat{Z})$ as modules over themselves, $\tau^{\text{ch}}$ intertwines the module structures in the sense that 
$$\tau^{\text{ch}}(\nu_m (\mu)) = (-1)^{|\nu|} (\phi^{\text{ch}}(\nu))_m (\tau^{\text{ch}}(\mu)), \ \text{for all} \ m \in \mathbb{Z}.$$ Here $\nu$ is one of the generators of $\Omega^{\text{ch},H}(Z)$ regarded as a vertex algebra, and $\mu \in \Omega^{\text{ch},H}(Z)$ regarded as a $\Omega^{\text{ch},H}(Z)$-module. 

By Theorem 2 of \cite{LM15}, the cohomology of $\Omega^{\text{ch},H}(Z)$ with respect to its twisted differential $D_H$ vanishes in positive weight, and coincides with the classical twisted cohomology in weight zero. In weight zero, $\tau^{\text{ch}}$ intertwines the differentials $D_H$ and $D_{\widehat{Z}, \widehat{H}} $ up to a sign, but unfortunately, this intertwining property no longer holds in positive weight. It is therefore not obvious that the inclusion of complexes $$(\mathcal{A}_n^{\bar{k}}(\widehat{Z}), \pi^*\nabla^{L^{\otimes n}}-\iota_{n{\widehat{v}}}+\widehat{H}) \hookrightarrow (\mathcal{A}_n^{\text{ch}, \widehat{H},\bar{k}}(\widehat{Z}), D_{\widehat{Z}, \widehat{H}})$$ induces an isomorphism in cohomology, although we expect this to be the case. In the last section, we will prove this in the special case where both circle bundles $Z$ and $\widehat{Z}$ are trivial, and the fluxes $H$ and $\widehat{H}$ are both zero.

Note that in the case $n = 0$, the isomorphism \eqref{intro:main} does {\it not} recover the chiral T-duality isomorphism of our previous paper \cite{LM15}, namely,
\begin{equation} \label{intro:old}
(\Omega^{\text{ch},H,\bar{k}}(Z))^{i\mathbb{R}[t]} / \langle L_A  - \iota_A H \rangle \rightarrow (\Omega^{\text{ch},\widehat{H},\overline{k+1}}(\widehat{Z}))^{i\mathbb{R}[t]} / \langle L_{\widehat A}  - \iota_{\widehat{A}} \widehat{H} \rangle.\end{equation} 
In particular, the $n=0$ term on the left side of \eqref{intro:main} is isomorphic to the $\TT$-invariant space $\Omega^{\text{ch},H,\bar{k}}(Z)^{\TT}$, which is larger than the left side of \eqref{intro:old}. Moreover, the right side of \eqref{intro:main} for $n=0$ is a different structure and is not a subquotient of the chiral de Rham complex of $\widehat{Z}$. The T-duality isomorphism \eqref{intro:main2} in this paper is stronger and more natural than the one in \cite{LM15} because on the left side the entire chiral de Rham complex appears rather than a subquotient. But the price we pay is that the object on the right side is a new kind of vertex algebra sheaf which incorporates sections of a line bundle $L$ on $\widehat{Z}$. This construction is very special since it makes use of the fact that $Z$ and $\widehat{Z}$ are $T$-dual to each other. An open question is whether it is possible to construct the exotic chiral de Rham complex on more general manifolds with line bundles, generalizing the construction given in this paper. \\

\section{Vertex algebras}\label{sect:vertex}

In this section, we define vertex algebras, which have been discussed from various points of view in the literature (see for example \cite{Borcherds,FLM,K,FBZ}). We will follow the formalism developed in \cite{LZ2} and partly in \cite{LiI}. Let $V=V_0\oplus V_1$ be a super vector space over $\mathbb{C}$, and let $z,w$ be formal variables. Let $\text{QO}(V)$ denote the space of linear maps $$V\ra V((z))=\{\sum_{n\in\mathbb{Z}} v(n) z^{-n-1}|
v(n)\in V,\ v(n)=0\ \text{for} \ n>\!\!>0 \}.$$ Each $a\in \text{QO}(V)$ can be represented as a power series
$$a=a(z)=\sum_{n\in\mathbb{Z}}a(n)z^{-n-1}\in \text{End}(V)[[z,z^{-1}]].$$ Each $a\in
\text{QO}(V)$ is assumed to be of the form $a=a^0+a^1$ where $a^i:V_j\ra V_{i+j}((z))$ for $i,j\in\mathbb{Z}/2\mathbb{Z}$, and we write $|a^i| = i$.

For all $n\in \mathbb{Z}$, $\text{QO}(V)$ has a bilinear operation defined on homogeneous elements $a,b$ by 
$$ a(w)_{(n)}b(w)=\text{Res}_z a(z)b(w)\ \iota_{|z|>|w|}(z-w)^n -
(-1)^{|a||b|}\text{Res}_z b(w)a(z)\ \iota_{|w|>|z|}(z-w)^n.$$
Here $\iota_{|z|>|w|}f(z,w)\in\mathbb{C}[[z,z^{-1},w,w^{-1}]]$ denotes the
power series expansion of a rational function $f$ in the region
$|z|>|w|$. For $a,b\in \text{QO}(V)$, we have the following identity of power series, known as the {\it operator product expansion} (OPE) formula.
 \begin{equation}\label{opeform} a(z)b(w)=\sum_{n\geq 0}a(w)_{(n)}b(w)\ (z-w)^{-n-1}+:a(z)b(w):. \end{equation}
Here $:a(z)b(w):\ =a(z)_-b(w)\ +\ (-1)^{|a||b|} b(w)a(z)_+$, where $$a(z)_-=\sum_{n<0}a(n)z^{-n-1},\qquad a(z)_+=\sum_{n\geq
0}a(n)z^{-n-1}.$$ We write
$$a(z)b(w)\sim\sum_{n\geq 0}a(w)_{(n)}b(w)\ (z-w)^{-n-1},$$ where
$\sim$ means equal modulo the term $:a(z)b(w):$, which is regular at $z=w$. 

Note that $:a(w)b(w):$ is a well-defined element of $\text{QO}(V)$. It is called the {\it normally ordered product} of $a$ and $b$, and it
coincides with $a_{(-1)}b$. The other negative products are given by
$$ n!\ a(z)_{(-n-1)}b(z)=\ :(\partial^n a(z))b(z):,\qquad \partial = \frac{d}{dz}.$$
For $a_1(z),\dots ,a_k(z)\in \text{QO}(V)$, the iterated normally ordered product is defined to be
\begin{equation}\label{iteratedwick} :a_1(z)a_2(z)\cdots a_k(z):\ =\ :a_1(z)b(z):,\qquad b(z)=\ :a_2(z)\cdots a_k(z):.\end{equation}
We often omit the variables $z,w$ when no confusion can arise.

We denote the constant power series $\text{Id}_V \in \text{QO}(V)$ by $1$. A subspace $\cA\subseteq \text{QO}(V)$ containing $1$ that is closed under all the above products will be called a {\it quantum operator algebra} (QOA). Elements $a,b\in \text{QO}(V)$ are called {\it local} if if $(z-w)^N [a(z),b(w)]=0$ for some $N\geq 0$. Here $[\cdot ,\cdot]$ denotes the super bracket. A {\it vertex algebra} is a QOA whose elements are pairwise local. This definition is well known to be equivalent to the notion of a vertex algebra in \cite{FLM}. 

A vertex algebra $\cA$ is {\it generated} by a subset $S=\{a_i|\ i\in I\}$ if every $a\in\cA$ can be written as a linear combination of nonassociative words in the letters $a_i$ for $i\in I$ and the above products for $n\in\mathbb{Z}$. We say that $S$ {\it strongly generates} $\cA$ if every $a\in\cA$ can be written as a linear combination of words in the letters $a_i$, and the above products for $n<0$. Equivalently, $\cA$ is spanned by \begin{equation} \label{norelations} \{ :\partial^{k_1} a_{i_1}\cdots \partial^{k_m} a_{i_m}:| \ i_1,\dots,i_m \in I,\ k_1,\dots,k_m \geq 0\}.\end{equation} 

A very useful description of a vertex algebra $\cA$ is a strong generating set $\{a_i|\ i\in I\}$ for $\cA$, together with a set of generators $\{b_k|\ k\in K\}$ for the ideal $\cI$ of relations among the generators and their derivatives, that is, all expressions of the form \eqref{norelations} that vanish. Given such a description, to define a homomorphism $\phi$ from $\cA$ to another vertex algebra $\cB$, it suffices to define $\phi(a_i)$ for $i\in I$ and show the following.
\begin{enumerate}
\item $\phi$ preserves pairwise OPEs among the generators; i.e., $\phi((a_i)_{(n)} a_j) = \phi(a_i)_{(n)}  \phi(a_j)$ for all $i,j \in I$ and $n\geq 0$.
\item $\phi(b_k) = 0$ for all $k \in K$.
\end{enumerate}
This will be our method of constructing vertex algebra homomorphisms in this paper.

A {\it conformal structure} on $\cA$ is an element $L(z) = \sum_{n\in \mathbb{Z}} L_n z^{-n-2} \in \cA$ satisfying $$L(z) L(w) \sim \frac{c}{2}(z-w)^{-4} + 2L(w) (z-w)^{-1} + \partial L(w)(z-w)^{-1},$$ such that $L_{-1}$ acts by $\partial$ on $\cA$ and $L_0$ acts diagonalizably. The constant $c$ is called the {\it central charge}, and the grading by $L_0$-eigenvalue is called {\it conformal weight}. In all our examples, the conformal weight grading is by the nonnegative integers. In the presence of a conformal weight grading, we always write a homogeneous element $a(z) =  \sum_{n\in Z} a(n) z^{-n-1}$ in the form 
\begin{equation} \label{shiftofindex} \sum_{n\in Z} a_n z^{-n-\text{wt}(a)},\qquad a_n = a(n+\text{wt}(a)-1). \end{equation} In this notation, for fields $a,b \in \cA$, we have $a_n b = a_{(n+\text{wt}(a)-1)} b$.

A {\it module} $\cM$ over a vertex algebra $\cA$ is a vector space $\cM$ together with a QOA homomorphism $\cA \ra \text{QO}(\cM)$. In particular, for each $a\in \cA$, we have a field $a_{\cM}(z) = \sum_{n\in \mathbb{Z}} a_{\cM}(n) z^{-n-1}$ where $a_{\cM}(z) \in \text{End}(\cM)$. If $\cA$ and $\cM$ are graded by conformal weight, we write $a_{\cM}(z) = \sum_{n\in \mathbb{Z}} a_{\cM,n} z^{-n-\text{wt}(a)}$, and we require that $a_{\cM,n}$ has weight $- n$.

\section{The chiral de Rham complex}\label{sect:cdr}

The chiral de Rham complex $\Omega^{\text{ch}}_Z$ is a sheaf of vertex algebras on any nonsingular algebraic variety $Z$, which was introduced by Malikov, Schechtman, and Vaintrob \cite{MSV1,MSV2}. As observed in \cite{MSV1}, a similar construction also works in the setting of smooth manifolds. However, the resulting object is no longer a sheaf, but instead is a {\it weak sheaf} in the terminology of \cite{LLS}. We briefly recall what this means. Suppose that we have a family of sheaves of vector spaces $\{\cF_n|~n=0,1,2,\dots\}$ on a smooth manifold $Z$. 
The direct sum $\cF$ defined by $\cF(U) = \bigoplus_{n\geq 0} \cF_n(U)$ for an open set $U\subseteq Z$, is a presheaf but not a sheaf. For example, in the case $Z = \mathbb{R}$ and each $\cF_n$ a copy of the structure sheaf $C^\infty$, if we cover $\mathbb{R}$ by an infinite collection of open intervals, one can use bump functions to construct a family of sections which are compatible on overlaps but do not give rise to a global section of $\cF$, that is, an element of the direct sum. However, $\cF$ does satisfy a slightly weaker version of the reconstruction axiom: 
$$ 0\ra\cF(U)\rightarrow\prod_i\cF(U_i)\rightrightarrows\prod_{i,j}\cF(U_i\cap U_j),$$
is exact for {\it finite} open covers $\{U_i\}$ of an open set $U$. Following \cite{LLS}, a weak sheaf is a presheaf which satisfies this weaker exactness condition.

If $Z$ is a smooth manifold, and $U\subseteq Z$ is any open set, $\Omega^{\text{ch}}(U)$ is an $\mathbb{N}$-graded vertex algebra by conformal weight, and we denote the conformal weight $n$ subspace by $\Omega^{\text{ch}}(U)[n]$. For each $n$, the assignment $U \mapsto \Omega^{\text{ch}}(U)[n]$ defines a sheaf of vector spaces on $Z$, and $\Omega^{\text{ch}}_Z$ is the weak sheaf of vertex algebras defined by $\Omega^{\text{ch}}(U) = \bigoplus_{n\geq 0} \Omega^{\text{ch}}(U)[n]$. Note that $\Omega^{\text{ch}}_Z$ is {\it not} the sheafification of this presheaf, which is too big to be a sheaf of vertex algebras. Similarly, the exotic chiral de Rham complex $\mathcal{A}^{{\rm ch}, \widehat{H} }_{\widehat{Z}}$ that we will construct has only a filtration \eqref{intro:filtration} by conformal weight. Each filtered component $\mathcal{A}^{{\rm ch}, \widehat{H}}(\widehat{Z})_{[i]}$ is an ordinary sheaf, and the union of these components is a weak sheaf. For simplicity, we will drop the word \lq\lq weak" throughout this paper.

For a coordinate open set $U\subseteq \mathbb{R}^n$ with coordinate functions $\gamma^1,\dots,\gamma^n$, the algebra of sections $\Omega^{\text{ch}}(U)$ has odd generators $b^i(z) = \sum_{n\in \mathbb{Z}} b^i_n z^{-n-1}$ and $c^i(z) = \sum_{n\in \mathbb{Z}} c^i_n z^{-n}$, even generators $\beta^i(z) = \sum_{n\in \mathbb{Z}} \beta^i_n z^{-n-1}$, as well as an even generator $f(z) = \sum_{n\in \mathbb{Z}} f_n z^{-n}$ for every smooth function $f = f(\gamma^1,\dots, \gamma^n) \in C^{\infty}(U)$. The field $\beta^i$ corresponds to the vector field $\frac{\partial}{\partial \gamma^i}$, $c^i$ corresponds to the $1$-form $d\gamma^i$, and $b^i$ corresponds to the contraction operator $\iota_{\partial/\partial \gamma^i}$. These fields satisfy the following nontrivial OPE relations
\begin{equation} \begin{split} & \beta^i(z) f(w) \sim \frac{\partial f}{\partial \gamma^i}(w)(z-w)^{-1},
\\ & b^i (z)c^j(w)\sim \delta_{i,j}(z-w)^{-1},\end{split}\end{equation}
which generalizes the formula $\beta^i(z)\gamma^j(w) \sim \delta_{i,j}(z-w)^{-1}$. These OPE relations define a Lie conformal algebra \cite{K}, and $\Omega^{\text{ch}}(U)$ is defined as the quotient of the corresponding universal enveloping vertex algebra by the ideal generated by
\begin{equation} \label{cdrrelations} \partial f - \sum_{i=1}^n :\frac{\partial f}{\partial x^i} \partial \gamma^i:,\qquad :fg: \ - fg, \qquad 1- \text{Id}.\end{equation}
A typical element of $\Omega^{\text{ch}}(U)$ is a linear combination of fields of form
\begin{equation} \label{cdr:typical} :f \partial^{a_1} b^{i_1} \cdots \partial^{a_r} b^{i_r} \partial^{d_1} c^{j_1} \cdots \partial^{d_s} c^{j_s} \partial^{e_1} \beta^{k_1} \cdots \partial^{e_t} \beta^{k_t} \partial^{m_1} \gamma^{l_1} \cdots \partial^{m_u} \gamma^{l_u}:,\end{equation} where $a_i, d_i, e_i \geq 0$ and $m_i\geq 1$. In particular, there are no nontrivial normally ordered relations among the $b^i, c^i, \beta^i, \partial \gamma^i$ and their derivatives, so the set of all Poincar\'e-Birkhoff-Witt monomials in these fields and their derivatives form a basis of $\Omega^{\text{ch}}(U)$ as a module over $C^{\infty}(U)$.

Now consider a smooth change of coordinates $g:U \ra U'$, $$\tilde{\gamma}^i = g^i(\gamma) = g^i(\gamma^1,\dots, \gamma^n),\qquad \gamma^i = f^i(\tilde{\gamma})= f^i(\tilde{\gamma}^1,\dots, \tilde{\gamma}^n).$$ We get the following transformation rules:
\begin{equation} \begin{split} \label{cdrlocaltransform} & \tilde{c}^i = \ :\frac{\partial g^i}{\partial \gamma^j} c^j:, \qquad  \tilde{b}^i = \ :\frac{\partial f^j}{\partial \tilde{\gamma}^i} (g(\gamma)) b^j:,
\\ & \tilde{\beta}^i =\ :\beta^j \frac{\partial f^j}{\partial \tilde{\gamma}^i}(g(\gamma)): + :\frac{\partial^2 f^k}{\partial \tilde{\gamma}^i \partial \tilde{\gamma}^l} (g(\gamma)) \frac{\partial g^l}{\partial \gamma^r} c^r b^k:.\end{split} \end{equation}
These new fields satisfy OPE relations
$$\tilde{b}^i(z) \tilde{c}^j(w) \sim \delta_{i,j}(z-w)^{-1},\qquad \tilde{\beta}^i(z) \tilde{f}(w) \sim \frac{\partial \tilde{f}}{\partial \tilde{\gamma}^i} (z-w)^{-1}.$$ Here $\tilde{f} = \tilde{f}(\tilde{\gamma}^1,\dots, \tilde{\gamma}^n)$ is any smooth function. Therefore $g:U \ra U'$ induces a vertex algebra isomorphism $ \phi_g: \Omega^{\text{ch}}(U) \rightarrow \Omega^{\text{ch}}(U')$. Moreover, given diffeomorphisms of open sets $U_1 \xrightarrow{g} U_2 \xrightarrow{h}U_3$, we get
$\phi_{h \circ g}=\phi_{g} \circ \phi_{h}$. This allows one to define the sheaf $\Omega^{\text{ch}}_Z$ on any smooth manifold $Z$.
Consider the following locally defined fields
\begin{equation} J = \sum_{i=1}^n :b^i c^i:,\qquad Q = \sum_{i=1}^n :\beta^i c^i:,\qquad G = \sum_{i=1}^n :b^i \partial\gamma^i:, \qquad L = \sum_{i=1}^n :\beta^i \partial \gamma^i: - :b^i \partial c^i:.\end{equation} These satisfy the OPE relations of a topological vertex algebra of rank $n$ \cite{LZ1}.
\begin{equation}\begin{split} 
L(z) L(w) & \sim 2L(w)(z-w)^{-2} + \partial L(w)(z-w)^{-1},
\\ L(z) J(w) & \sim -n (z-w)^{-3} + J(w) (z-w)^{-2} +  \partial J(w) (z-w)^{-1}, 
\\ L(z) G(w) & \sim  2G(w) (z-w)^{-2} +  \partial G(w) (z-w)^{-1}, 
\\ L(z) Q(w) & \sim  Q(w) (z-w)^{-2} +  \partial Q(w) (z-w)^{-1}, 
\\ J(z) J(w) & \sim - n (z-w)^{-2}, \qquad G(z) G(w)  \sim 0, \qquad  Q(z) Q(w) \sim 0,
\\ J(z) G(w) & \sim -G(w) (z-w)^{-1}, \qquad J(z) Q(w) \sim Q(w) (z-w)^{-1},
\\ Q(z) G(w) & \sim n (z-w)^{-3} + J(w) (z-w)^{-2} + L(w)(z-w)^{-1}.
\end{split}
\end{equation}
Under $g:U \ra U'$, these fields transform as
\begin{equation} \label{transform:LJG} \begin{split} 
& \tilde{L} =  L,\qquad \tilde{G} = G,
\\ & \tilde{J} = J + \partial \bigg(\text{Tr}\ \text{log}\bigg(\frac{\partial g^i}{\partial b^j}\bigg)\bigg),\qquad 
\tilde{Q} = Q +  \partial \bigg( \frac{\partial}{\partial \tilde{b}^r} \bigg(\text{Tr}\ \text{log} \bigg( \frac{\partial f^i}{\partial \tilde{b}^j} \bigg) \bigg) \tilde{c}^r \bigg),
\end{split}\end{equation}
Therefore $L$ and $G$ are globally defined on any manifold $Z$. Although $J$ and $Q$ are not globally defined in general, the operators $J_0$ and $Q_0$ are well-defined. Note that $\Omega^{\text{ch}}(Z)$ has a bigrading by degree and weight, where the weight is the eigenvalue of $L_0$ and degree is the eigenvalue of $J_0$. Also, $Q_0$ is a square-zero operator and we define the differential $D$ to be $Q_0$. It is {\it vertex algebra derivation}, that is, a derivation of all vertex algebra products, and it coincides with the de Rham differential at weight zero. Note that $G_0$ is a contracting homotopy for $D$, i.e., $[D,G_0] = L_0$. This shows that the cohomology $H^*(\Omega^{\text{ch}}(Z),D)$ vanishes in positive weight. Each $f$ has weight $0$ and degree $0$, $c^i$ has weight $0$ and degree $1$, $\beta^i$ has weight $1$ and degree $0$, and $b^i$ has weight $1$ and degree $-1$. Therefore the weight zero component of $\Omega^{\text{ch}}(Z)$ is just $\Omega(Z)$, and the embedding $\Omega(Z) \hookrightarrow \Omega^{\text{ch}}(Z)$ induces an isomorphism in cohomology.

\section{Coordinate-free description} 
For any open set $U\subseteq Z$, we may regard $f\in C^{\infty}(U)$ and $\omega \in \Omega^1(U)$ as sections of $\Omega^{\text{ch}}(U)$ of weight zero and degrees $0$ and $1$, respectively. Given a vector field $X \in {\rm Vect}(U)$, there are sections $$\iota_X(z) = \sum_{n\in \mathbb{Z}} (\iota_X)_n z^{-n-1},\qquad L_X(z) = \sum_{n\in \mathbb{Z}} (L_X)_n z^{-n-1}$$ in $\Omega^{\text{ch}}(U)$ of weight $1$ and degrees $-1$ and $0$, respectively, and the local description of $\iota_X$ and $L_X$ is given in \cite{LL}. Let $\gamma^1,\dots, \gamma^n$ be local coordinates and $X = \sum_{i=1}^n f_i  \frac{\partial}{\partial \gamma^i}$ where each $f_i = f_i (\gamma^1,\dots, \gamma^n)$ is a smooth function. Then \begin{equation} \iota_X = \sum_{i=1}^n :f_i b^i:,\qquad L_X = D(\iota_X) = \sum_{i=1}^n : \beta^i f_i: + \sum_{i=1}^n \sum_{j=1}^n : \frac{\partial f_j}{\partial \gamma^i} c^i b^j:.\end{equation} The next theorem\footnote{The coordinate-free description of the relations is due to Bailin Song, and we thank him for sharing it with us.} gives a useful coordinate-independent description of $\Omega^{\text{ch}}(U)$ when $U$ is a coordinate open set.

\begin{theorem}  For a coordinate open set $U\subseteq \mathbb{R}^n$, $\Omega^{\text{ch}}(U)$ is strongly generated by the following fields: \begin{equation} \label{coordindep} f \in  C^{\infty}(U),\qquad \omega \in \Omega^1(U),\qquad L_X, \iota_X, \qquad X\in {\rm Vect}(U).\end{equation} These satisfy the following OPE relations.
\begin{equation}\label{opecdr} \begin{split}
& \iota_X(z) \iota_Y(w) \sim 0,
\\ & L_X(z) \iota_Y(w) \sim \iota_{[X,Y]}(w)(z-w)^{-1}, \qquad L_X(z) L_Y(w) \sim L_{[X,Y]}(w)(z-w)^{-1},
\\ & L_X(z) \omega(w) \sim {\rm Lie}_X(\omega)(w)(z-w)^{-1}, \qquad \iota_X(z) \omega(w) \sim \iota_X(\omega)(w)(z-w)^{-1},
\\ & L_X(z) f(w) \sim X(f)(w)(z-w)^{-1}, \qquad \iota_X(z) f(w) \sim 0.
\end{split}\end{equation}
The ideal of normally ordered relations among these fields is generated by the following elements.
\begin{equation} \label{cirelations}  \begin{split}
& 1- {\rm Id},\qquad  :f g: - fg, \qquad  :\nu \omega: -  \nu\omega , \qquad f,g \in C^{\infty}(U), \qquad  \nu,\omega \in \Omega^1(U),
\\ & \iota_{gX}-: g\iota_X:,\qquad L_{gX}-: (dg) \iota_X:-:gL_X:, \qquad g \in C^{\infty}(U),\qquad X \in {\rm Vect}(U),
\\ & \partial g(\phi_1,\dots,\phi_n)-\sum_{i=1}^n  \frac{\partial g}{\partial{\phi_i}} \partial \phi_i ,\qquad g\in C^{\infty}(\mathbb{R}^n),\qquad \phi_i \in C^{\infty}(U).
\end{split}\end{equation}
\end{theorem}

\begin{proof} For a coordinate open set with coordinates $\gamma^1,\dots, \gamma^n$, \eqref{coordindep} is a strong generating set for $\Omega^{\text{ch}}(U)$ since it contains the above generators $f \in C^{\infty}(U), b^i, c^i, \beta^i$ as a subset. Similarly, the set of relations \eqref{cirelations} are all consequences of the set \eqref{cdrrelations}, which is a subset of \eqref{cirelations}.
\end{proof}

We call an open set $U\subseteq Z$ {\it small} if $\Omega^{\text{ch}}(U)$ has the strong generating set \eqref{coordindep}. We call an open cover $\{U_{\alpha}\}$ of $Z$ a {\it small open cover} if each $U_{\alpha}$ is small. Aside from coordinate open sets, there is another type of small open set that will be useful. These are of the form $U \times \TT^m$ where $U$ is a coordinate open set, and $\TT^m$ is a torus of rank $m$. The reason such a set is small is that if $y^1,\dots, y^m$ are coordinates on $\TT^m$ defined up to shifts by $2\pi i k$ for $k\in \mathbb{Z}$, the corresponding fields $\partial y^i$, $c^i = d y^i$, $\beta^i$, and $b^i$, are globally defined. If $\pi: Z \ra M$ is a principal circle bundle, we often choose a trivializing open cover $\{V_{\alpha}\}$ for $M$ such that each $V_{\alpha}$ is a coordinate open set. Then $\{U_{\alpha} = \pi^{-1}(V_{\alpha})\}$ is a small open cover for $Z$, and each $U_{\alpha} \cong V_{\alpha} \times \TT$.

Even though $\Omega^{\text{ch}}(U)$ contains the ring of smooth functions $C^{\infty}(U)$ as the weight zero subspace, it is not a $C^{\infty}(U)$-module because of the nonassociativity of the normally ordered product. In other words, for $f,g \in C^{\infty}(U)$ and $\nu \in \Omega^{\text{ch}}(U)$, 
$$ :(fg)\nu:\  - \ :fg\nu:\ 
=  \sum_{n\geq 0}\frac{1}{(n+1)!}\big( :(\partial^{n+1} f)(g_{(n)} \nu):\ + (\partial^{n+1} g)(f_{(n)} \nu):\big)\ ,$$ and the right hand side need not vanish. 
However, $\Omega^{\text{ch}}(U)$ is a loop module over $C^{\infty}(U)$ in the sense of \cite{Borisov}, and $\Omega^{\text{ch}}_Z$ is a sheaf of loop modules over the structure sheaf $C^{\infty}$. For practical purposes it can be treated like an ordinary sheaf of $C^{\infty}$-modules since global sections can be constructed by gluing local sections using a partition of unity. We thank B. Song for explaining this to us.

\begin{remark} \label{rem:cdrlocal} Given a sheaf of vertex algebras on a manifold $M$ which is a sheaf of $C^{\infty}$-loop modules, a local but coordinate-independent description is useful for the following reason. To specify a homomorphism between two such sheaves $\cA_M \rightarrow \cB_M$, it is enough to give a vertex algebra homomorphism $\phi_{\alpha}: \cA(U_{\alpha}) \rightarrow \cB(U_{\alpha})$ which intertwines the $C^{\infty}$-loop module structures, such that $\phi_{\alpha}$, $\phi_{\beta}$ agree on the overlap $U_{\alpha} \cap U_{\beta}$. If we have coordinate-independent generators and relations for $\cA_M$ and $\cB_M$, it suffices to show that the OPEs among the generators are preserved and the ideal of relations is annihilated; the agreement on overlaps is then automatic. This applies to morphisms of sheaves of modules over such vertex algebra sheaves as well. \end{remark}

\subsection{$H$-twisted chiral de Rham complex}
Suppose that $H$ is a closed $3$-form on $Z$. Recall from \cite{LM15} that for a coordinate open set $U$, $\Omega^{\text{ch},H}(U)$ has strong generators  $\tilde{L}_X, \tilde{\iota}_X(z), \tilde{f}, \tilde{\omega}$ satisfying
\begin{equation} \label{opecdrtwist}
\begin{split}
\tilde{\iota}_X(z) \tilde{\iota}_Y(w) & \sim 0,
\\ \tilde{L}_X(z) \tilde{\iota}_Y(w) & \sim \big(\tilde{\iota}_{[X,Y]}(w) + (\widetilde{\iota_X \iota_Y H})(w)\big)(z-w)^{-1},
\\ \tilde{L}_X(z) \tilde{L}_Y(w) & \sim \big(\tilde{L}_{[X,Y]}(w) +( \widetilde{D \iota_X \iota_Y H})(w)\big)(z-w)^{-1},
\\ \tilde{L}_X(z) \tilde{\omega}(w) & \sim \widetilde{{\rm Lie}_X(\omega)}(w)(z-w)^{-1},\qquad \tilde{\iota}_X(z) \tilde{\omega}(w)  \sim \widetilde{\iota_X(\omega)}(w)(z-w)^{-1},
\\ \tilde{L}_X(z) \tilde{f}(w) & \sim \widetilde{X(f)}(w)(z-w)^{-1}, \qquad \tilde{\iota}_X(z) \tilde{f}(w)  \sim 0.
\end{split} \end{equation}
Note that $\iota_X \iota_Y H$ is a one-form $\nu \in \Omega^1(U)$, and the notation $\widetilde{\iota_X \iota_Y H}$ means $\tilde{\nu}$, and similarly for the other uses of the wide tilde notation above. The ideal of relations among these fields has the same generating set \eqref{cirelations} as the untwisted case, where each field is replaced by the tilde version. The corresponding vertex algebra sheaves are all isomorphic to the {\it untwisted} chiral de Rham sheaf.

\begin{theorem} [\cite{LM15}, Theorem 3] \label{untwisting} Let $\{U_{\alpha}\}$ be a small open cover of $Z$. Define a map $\Omega^{\rm ch}(U_{\alpha})\ra \Omega^{{\rm ch},H}U_{\alpha})$ by \begin{equation} \label{eqn:untwisting} \iota_X \mapsto \tilde{\iota}_X,\qquad L_X \mapsto \tilde{L}_X - \widetilde{\iota_X H},\qquad f \mapsto \tilde{f},\qquad \omega \mapsto \tilde{\omega}.\end{equation}This is an isomorphism of vertex algebras for each $U_{\alpha}$, and it defines a sheaf isomorphism $\Omega^{{\rm ch}}_Z \cong \Omega^{{\rm ch},H}_Z$.
\end{theorem}

For the rest of this paper, we will work with the twisted version $\Omega^{\text{ch},H}(U)$, and for simplicity of notation we shall drop the tilde symbols. Note that the chiral de Rham differential $D$ acts on the generators of $\Omega^{\text{ch},H}(U)$ as follows:
$$D(f) = df,\qquad D(\omega) = d\omega,\qquad D(\iota_X) = L_X - \iota_X H,\qquad D(L_X) = L_X H.$$

\begin{remark} \label{rem:homogeneous} Note that in $\Omega^{\text{ch},H}(U)$, the generators $f, \omega$ are homogeneous of weight zero, and $\iota_X$ is homogeneneous of weight one. However, $L_X$ is not homogeneous with respect to the conformal weight grading, but must be replaced with the element $L_X - \iota_X H$, which is homogeneous of weight $1$.
\end{remark}

\section{Fourier decomposition}
Suppose now that $Z$ is a principal circle bundle over $M$, with circle denoted by $\mathbb{T}$, which we denote by $\pi: Z \rightarrow M$. Let $H$ be an integral closed $3$-form on $Z$. By averaging over $\TT$, we may assume without loss of generality that $H$ is $\TT$-invariant since this does not change the cohomology class $[H]$. Let $\nu$ denote the vector field infintesimally generated by $\mathbb{T}$, and fix a connection form $A \in \Omega^1(Z)$, normalized so $\iota_{\nu} A= 1$. By abuse of notation, we often denote $\iota_{\nu}$ by $\iota_A$ in order to emphasize the duality between the vector field and connection form. We will denote the even and odd fields in $\Omega^{\text{ch},H}(Z)$ corresponding to $\nu$ by $L_A$ and $\iota_A$, respectively.

Since $H$ is $\mathbb{T}$-invariant, we may write $H = H^{3} + A \wedge H^{2}$ where $H^{3}, H^{2}$ are basic forms, that is, elements of $\pi^*(\Omega(M))$. For each open set $U \subseteq Z$, define $$\Omega^{\text{ch},H}_n(U) = \{\alpha \in \Omega_{\text{ch},H}(U)|\ (L_A)_{0}(\alpha) = n \alpha\}.$$ Then $$\Omega^{\text{ch},H}(U) \cong \bigoplus_{n\in \mathbb{Z}} \Omega^{\text{ch},H}_n(U),$$ where the direct sum denotes the Fr\'echet space completion of the ordinary direct sum. Choose a trivializing open cover $\{V_{\alpha}\}$ for $M$ such that each $V_{\alpha}$ is a coordinate open set. Then $\{U_{\alpha} = \pi^{-1}(V_{\alpha})\}$ is a small open cover for $Z$, and each $U_{\alpha} \cong V_{\alpha} \times \TT$.

Note that $\Omega^{\text{ch},H}_0(U_{\alpha}) \cong \Omega^{\text{ch},H}(U_{\alpha})^{\mathbb{T}}$ and each weight space $\Omega^{\text{ch},H}_n(U_{\alpha})$ for the action of $\TT$ is a module over $\Omega^{\text{ch},H}_0(U_{\alpha})$. Moreover, $\Omega^{\text{ch},H}_0(U_{\alpha})$ has the strong generating set 
$$\{\iota_X, L_X, f, \omega, A, \Gamma^A|\ X \in {\rm Vect}_{\text{hor}}(U_{\alpha}),\ f \in \pi^*(C^{\infty}(V_{\alpha})),\ \omega \in \pi^*(\Omega^1(V_{\alpha}))\},$$
described in \cite{LM15}. In this notation, ${\rm Vect}_{\text{hor}}(U_{\alpha}) = \{X\in \text{Vect}(U_{\alpha})|\ \iota_X(A) = 0\}$ is the set of horizontal vector fields, and $\Gamma^A = G_{(0)} A = G_{(1)} \partial A$, which has degree zero and weight $1$. Note that $G$ has weight $2$, so in our earlier notation \eqref{shiftofindex} this is written as $\Gamma^A = G_{-1} A = G_{0} \partial A$. Recall that $D \Gamma^A = \partial A -\xi^A$, where $\xi^A$ has degree $1$, weight $1$, and satisfies $D \xi^A = \partial DA = D \partial A$. Also, recall that $\xi^A$ lies in the subalgebra of $\Omega^{\text{ch},H}(U_{\alpha})$ generated by $\pi^*(\Omega(V_{\alpha}))$, and in particular commutes with both $\iota_A$ and $L_A$. For convenience we recall the OPEs among the generators of $\Omega^{\text{ch},H}_0(U)$.
\begin{equation} \begin{split} \label{cdr:opealgebra}
 & L_X(z) \iota_Y(w)  \sim \big(\iota_{[X,Y]} + \iota_X \iota_Y H^3 + :A (\iota_X \iota_Y H^2): + :(\iota_X \iota_Y \widehat{H}^{2}) \iota_A:\big)(z-w)^{-1},
 \\ & L_X(z) L_Y(w)  \sim \big(L_{[X,Y]}+ L_X \iota_Y H^3 - \iota_X L_Y H^3 + :\widehat{H}^2 (\iota_X \iota_Y H^2): - :A (L_X \iota_Y H^2):
 \\ & + :A (\iota_X L_Y H^2): + :(L_X \iota_Y \widehat{H}^2) \iota_A: - :(\iota_X L_Y \widehat{H}^2) \iota_A:
 + : L_A (\iota_X \iota_Y \widehat{H}^{2}):\big)(w)(z-w)^{-1},
\\ & L_X(z) \omega(w)  \sim {\rm Lie}_X(\omega)(w)(z-w)^{-1}, \qquad  L_X(z) f(w)  \sim X(f)(w)(z-w)^{-1},
\\ & \iota_X(z) \omega(w)  \sim (\iota_X \omega)(w)(z-w)^{-1}, \qquad  \iota_X(z) f(w)  \sim 0,
\\ & L_X(z) A(w)  \sim (\iota_X \widehat{H}^2)(w)(z-w)^{-1}, \qquad L_X(z) \iota_A(w) \sim (\iota_X H^2)(w)(z-w)^{-1},
\\ & L_X(z) \Gamma^A(w)  \sim -(\iota_X \xi^A)(w)(z-w)^{-1}, \qquad \iota_X (z) \Gamma^A(w) \sim 0,
\\ & L_A(z) \Gamma^A(w) \sim (z-w)^{-2}, \qquad  \iota_A(z) A(w) \sim (z-w)^{-1},
\\ & L_A(z) \iota_X(w) \sim - (\iota_X H^2)(w)(z-w)^{-1},\qquad L_A(x) L_X(w) \sim - (\iota_X H^2)(w)(z-w)^{-1}.
 \end{split}
\end{equation}
Elements of $\Omega^{\text{ch},H}_n(U)$ may be described locally as follows. If $\theta_{\alpha}$ is a coordinate on $\TT$ which is defined up to shifts by $2\pi i k$ for $k\in \mathbb{Z}$, then the function $e^{n\theta_{\alpha}}$ in local coordinates lies in $\Omega^{\text{ch},H}_n(U_{\alpha})$; moreover, elements of $\Omega^{\text{ch},H}_n(U)$ are all of the form $:e^{n \theta_{\alpha}}\alpha:$ for some $\alpha \in \Omega^{\text{ch},H}_0(U_{\alpha})$. We have the following additional OPE relations.
\begin{equation} \label{cdr:additional} \begin{split}
& L_X (z) e^{n \theta_{\alpha}}(w)  \sim 0,\qquad \iota_X (z) e^{n \theta_{\alpha}}(w) \sim 0,
\\ & L_A(z) e^{n \theta_{\alpha}}(w) \sim  n e^{n \theta_{\alpha}}(w)(z-w)^{-1}, \qquad  \iota_A(z) e^{n \theta_{\alpha}}(w) \sim 0.
\\ & e^{n \theta_{\alpha}}(z) e^{m \theta_{\alpha}}(w) \sim 0, \ \text{for all}\ n,m.
\end{split} \end{equation}
These follow from the OPE relations \eqref{cdr:opealgebra} in $\Omega^{\text{ch},H}(U_{\alpha})$.

\section{Exotic twisted chiral de Rham complex}
As above, let $\pi: Z \rightarrow M$ be a principal $\TT$-bundle with flux form $H$, which we may assume to be $\TT$-invariant, and let $A \in \Omega^1(Z)$ be a connection form normalized so that $\iota_A A = 1$. Recall that $\iota_A$ means the contraction $\iota_{\nu}$ along the vector field $\nu$ infinitesimally generated by $\TT$. Let $\widehat{\pi}: \widehat{Z} \rightarrow M$ be the T-dual principal $\widehat{\TT}$-bundle with $\widehat{\TT}$-invariant flux form $\widehat{H}$, and fix a connection form $\widehat{A}\in \Omega^1(\widehat{Z})$ normalized so that $\iota_{\widehat{A}} \widehat{A} = 1$. Again, by abuse of notation $\iota_{\widehat{A}}$ means the contraction $\iota_{\widehat{\nu}}$ along the vector field $\widehat{\nu}$ infinitesimally generated by $\widehat{\TT}$.

Next, fix an open cover $\{V_{\alpha}\}$ for $M$ which trivializes both circle bundles, such that each $V_{\alpha}$ is a coordinate open set on $M$. Then $\{U_{\alpha} = \pi^{-1}(V_{\alpha})\}$ and $\{\widehat{U}_{\alpha} = \widehat{\pi}^{-1}(V_{\alpha})\}$ are small open covers for $Z$ and $\widehat{Z}$, respectively. Since $H$ is $\TT$-invariant, it can be written in the form $H=H^{3} + A \wedge H^{2}$ where $H^{3} \in \pi^* (\Omega^3(M))$ and $H^{2}\in \pi^*(\Omega^2(M))$. Similarly, since $\widehat{H}$ is $\widehat{\TT}$-invariant it can written as $\widehat{H} = \widehat{H}^{3} + \widehat{A} \wedge \widehat{H}^{2}$ where $\widehat{H}^{3} \in \widehat{\pi}^*(\Omega^3(M))$ and $\tilde{H}^{2} \in \widehat{\pi}^*(\Omega^2(M))$. By Equations (1.10) and (1.11) of \cite{BEM}, we can assume that $$H^{3} = \widehat{H}^{3},\qquad H^{2} =  d \widehat{A} = F_{\widehat{A}},\qquad \widehat{H}^{2} = dA = F_A,$$ where $F_A$ and $F_{\widehat{A}}$ denote the curvature forms associated to $A$ and $\widehat{A}$.

Next, let $L$ be the line bundle on $M$ associated to the circle bundle $Z$. We may write the connection form $A \in \Omega^1(Z)$ locally in the form $$A_{\alpha} = A_{\alpha,\text{bas}} + d \theta_{\alpha},$$ where $A_{\alpha,\text{bas}}$ is a basic $1$-form, and hence can be identified with an element of $\Omega^1(V_{\alpha})$. By abuse of notation, we denote this element by $A_{\alpha,\text{bas}}$ as well. 

For a local section $g$ of $L$ over $V_{\alpha}$, we can regard $g$ as a function $g: V_{\alpha}  \rightarrow \mathbb{R}$, and we have the covariant derivative $$\nabla^{L} (g)  = d g + A_{\alpha, \text{bas}} \wedge g.$$ Here $d$ is the de Rham differential on $M$. Finally, we fix a local nonvanishing section $s_{\alpha}$ which is constant along $V_{\alpha}$. For each $n\in \mathbb{Z}$, the $n$th tensor power $L^{\otimes n}$ has connection form locally given by $n A_{\alpha} = n A_{\alpha,\text{bas}} + n d \theta_{\alpha}$, and given a local section $g$ of $L^{\otimes n}$ over $M$, we have 
\begin{equation} \label{nabla} \nabla^{L^{\otimes n}} (g)  = d g + n A_{\alpha, \text{bas}} \wedge g.\end{equation} Also, $s_{\alpha}^n$ is a locally constant nowhere vanishing section of $L^{\otimes n}$. We use the same notation $s_{\alpha}^n$ to denote the section $\widehat{\pi}^*(s_{\alpha}^n)$ of $\widehat{\pi}^*(L^{\otimes n})$ over $\widehat{Z}$, when no confusion can arise.

We now define the exotic $\widehat{H}$-twisted chiral de Rham sheaf $\mathcal{A}^{\text{ch}, \widehat{H}}_{\widehat{Z}}$ on $\widehat{Z}$. We first define it locally by writing strong generators, OPE relations among the generators, and specifying the ideal of normally ordered relations among the generators. For each $V_{\alpha}$, we then write down an explicit isomorphism $$\Omega^{\text{ch}, H}_n(U_{\alpha}) \rightarrow \mathcal{A}^{\text{ch}, \widehat{H}}_{-n}(\widehat{U}_{\alpha}).$$ This is enough to get the isomorphism of vertex algebra sheaves on $M$, 
$$\pi_*(\Omega^{\text{ch}, H}_Z) \rightarrow \widehat{\pi}_*(\mathcal{A}^{\text{ch}, \widehat{H}}_{\widehat{Z}}).$$

Recall the set of horizontal vector fields ${\rm Vect}_{\text{hor}}(\widehat{U}_{\alpha}) = \{X\in \text{Vect}(\widehat{U}_{\alpha})|\ \iota_X(\widehat{A}) = 0\}$. First, for $n = 0$ we declare that $\mathcal{A}^{\text{ch}, \widehat{H}}_0(\widehat{U}_{\alpha})$ has strong generators 
$$\{L_X, \iota_X, \widehat{A}, \iota_{\widehat{A}}, L_A, \Gamma^A, f, \omega| \ X \in {\rm Vect}_{\text{hor}}(\widehat{U}_{\alpha}),\ f \in \widehat{\pi}^*(C^{\infty}(M)),\ \omega \in \widehat{\pi}^*(\Omega^1(M))\},$$ which satisfy OPE relations
\begin{equation} \label{ope:exoticcdr} \begin{split}
L_X(z) \iota_Y(w) & \sim \big(\iota_{[X,Y]} + \iota_X \iota_Y H^3 + :\widehat{A} (\iota_X \iota_Y \widehat{H}^2):    \big)(z-w)^{-1},
\\ L_X(z) L_Y(w) & \sim \big(L_{[X,Y]}  + L_X \iota_Y H^3 - \iota_X L_Y H^3 + : \widehat{H}^{2} (\iota_X \iota_Y H^{2}): + :H^2 (\iota_X \iota_Y \widehat{H}^2):
\\ & + :(L_X \iota_Y \widehat{H}^2) \widehat{A} : -  : (\iota_X L_Y \widehat{H}^2)\widehat{A}: +:L_A (\iota_X \iota_Y \widehat{H}^2): \big)(w)(z-w)^{-1},
\\ L_X(z) \omega(w) & \sim {\rm Lie}_X(\omega)(w)(z-w)^{-1}, \qquad   L_X(z) f(w)  \sim X(f)(w)(z-w)^{-1},
\\ \iota_X(z) \omega(w)  & \sim (\iota_X \omega)(w)(z-w)^{-1}, \qquad  \iota_X(z) f(w)  \sim 0,
\\ L_X(z) \widehat{A}(w) & \sim 0, \qquad L_X(z) \iota_{\widehat{A}}(w) \sim (\iota_X \widehat{H}^2)(w)(z-w)^{-1},
\\ L_A(z) \Gamma^A(w) & \sim (z-w)^{-2}, \qquad \iota_{\widehat{A}}(z) \widehat{A}(w) \sim (z-w)^{-1},
\\ L_X(z) \Gamma^A(w)  & \sim -(\iota_X \xi^A)(w)(z-w)^{-1}, \qquad \iota_X (z) \Gamma^A(w) \sim 0,
\\ L_A(z) \iota_X(w) & \sim 0, \qquad  L_A(x) L_X(w) \sim 0.
\end{split} \end{equation} 
The ideal of relations among these fields has the same generating set \eqref{cirelations}.

It is not immediately apparent that this structure defined by writing down generating fields and specifying OPE relations and normally ordered relations, leads to a vertex algebra. There is a general method for constructing vertex algebras starting from fields and OPE relations that is given by De Sole and Kac in \cite{DSK} in the language of $\lambda$-brackets, and it is translated into the language of OPEs in \cite{L}. Briefly, the universal enveloping vertex algebra associated to an OPE algebra can always be defined, although it may be trivial. In our case, the universal enveloping vertex algebra associated to the OPE algebra given by \eqref{ope:exoticcdr} is freely generated by these fields since in the notation of \cite{L}, all Jacobi identities (2.10) hold as consequences of equations (2.6)-(2.9) of \cite{L}. Therefore $\mathcal{A}^{\text{ch}, \widehat{H}}_0(\widehat{U}_{\alpha})$ is well-defined as a quotient of this structure by the relations generated by  \eqref{cirelations}, and in particular is a vertex algebra.

\begin{lemma} \label{lem:degreezeroiso} For each index $\alpha$, define a map $\phi^{\rm ch}_0: \Omega^{{\rm ch}, H}_0(U_{\alpha}) \rightarrow \mathcal{A}^{{\rm ch}, \widehat{H}}_0(\widehat{U}_{\alpha})$ by
\begin{equation} \begin{split} \label{tau0} & f \mapsto f, \qquad \omega \mapsto \omega,\qquad L_X \mapsto L_X -  :\iota_{\widehat{A}} (\iota_X H^{2}):, \qquad \iota_X \mapsto \iota_X,  
\\ & A \mapsto \iota_{\widehat{A}},\qquad \iota_A \mapsto \widehat{A},\qquad L_A \mapsto L_A +H^{2},\qquad \Gamma^A \mapsto \Gamma^A.\end{split} \end{equation} This map preserves OPE relations as well as the ideal of relations, so it determines a vertex algebra isomorphism. Moreover, $\phi^{{\rm ch}}_0$ induces an isomorphism of sheaves of vertex algebras on $M$,
\begin{equation}  \phi^{{\rm ch}}_0: \pi_* \big((\Omega^{{\rm ch}, H}_0)_{Z}\big) \rightarrow \widehat{\pi}_* \big((\mathcal{A}^{{\rm ch}, \widehat{H}}_0)_{\widehat{Z}}\big),\end{equation} Taking global sections, we get a vertex algebra isomorphism
\begin{equation}  \phi^{\rm ch}_0: \Omega^{{\rm ch}, H}_0(Z) \rightarrow \mathcal{A}^{{\rm ch}, \widehat{H}}_0(\widehat{Z}).\end{equation}
\end{lemma}

\begin{proof}
The fact that the map $\phi^{\rm ch}_0$ given by \eqref{tau0} preserves OPE relations is straightforward to verify using the OPE relations \eqref{cdr:opealgebra} and \eqref{ope:exoticcdr}. It is surjective since it takes generators to generators. To see that $\phi^{\text{ch}}_0$ is injective, recall that $V_{\alpha}$, $U_{\alpha}$, and $\widehat{U}_{\alpha}$ are small open sets. Therefore we may choose local coordinates such that $\Omega^{\text{ch}, H}_0(U_{\alpha})$ and $\mathcal{A}^{\text{ch}, \widehat{H}}_0(\widehat{U}_{\alpha})$ both admit bases consisting of Poincar\'e-Birkhoff-Witt monomials in the coordinate one-forms, contraction operators, vector fields, and the derivatives of coordinate functions as in \eqref{cdr:typical}, as modules over $C^{\infty}(V_{\alpha})$. Clearly $\phi^{\text{ch}}_0$ maps a basis to a basis, so it must be injective. Finally, the fact that $\phi^{\text{ch}}_0$ induces an morphism of vertex algebra sheaves on $M$ (which then must be an isomorphism), follows from Remark \ref{rem:cdrlocal}. \end{proof}

\begin{remark} Recall the vertex algebra
\begin{equation} \label{oldalgebra} \bigg(\Omega^{\text{ch}, \widehat{H}} (\widehat{U}_{\alpha})^{ i\mathbb{R}[t]} / \langle L_{\widehat{A}} - \widehat{H}^2 \rangle \bigg) \otimes \mathcal{H}(2),\end{equation} 
defined in \cite{LM15}, where $\mathcal{H}(2)$ is the rank $2$ Heisenberg vertex algebra with generators $L_A, \Gamma^A$ satisfying 
$$L_A(z) \Gamma^A(w) \sim (z-w)^{-2}.$$ The generators of $\mathcal{A}^{\text{ch}, \widehat{H}}_0(\widehat{U}_{\alpha})$ are the same as the generators of \eqref{oldalgebra} but the OPE algebras are different. So these structures coincide as vector spaces but not as vertex algebras.\end{remark}

\begin{remark} $\mathcal{A}^{\text{ch}, \widehat{H}}_0(\widehat{U}_{\alpha})$ has an action of $i \mathbb{R}[t]$ given by the modes $\{(L_A)_{(k)}|\ k\geq 0\}$, and the space $\mathcal{A}^{\text{ch}, \widehat{H}}_0(\widehat{U}_{\alpha})^{i\mathbb{R}[t]}$ is the subalgebra generated by the above generators except for $\Gamma^A$. 
\end{remark}

Next, for each $n \neq 0$, we define $\mathcal{A}^{\text{ch}, \widehat{H}}_n(\widehat{U}_{\alpha})$ to be a module over $\mathcal{A}^{\text{ch}, \widehat{H}}_0(\widehat{U}_{\alpha})$ with generator $s^n_{\alpha}$, which commutes with all generators of $\mathcal{A}^{\text{ch}, \widehat{H}}_0(\widehat{U}_{\alpha})$ except for $L_A$, and satisfies
\begin{equation} \label{ope:exotic} L_A(z) s^n_{\alpha}(w) \sim -n s^n_{\alpha}(w)(z-w)^{-1}.\end{equation}
Additionally, we declare that for all $n,m \neq 0$,
\begin{equation} \label{additional:exotic} \begin{split}  s^n_{\alpha} (z)  s^m_{\alpha}(w) & \sim 0,
\\ :s^n_{\alpha} s^m_{\alpha}: \  & = s^{n+m}_{\alpha},
\\ \partial s^n_{\alpha} & = - n :s^n_{\alpha}(\partial A - \partial A_{\alpha, \text{bas}}):.
\end{split} \end{equation} It follows that any element of $\mathcal{A}^{\text{ch}, \widehat{H}}_n(\widehat{U}_{\alpha})$ can be expressed in the form $:s^n_{\alpha} \eta:$ for some $\eta \in \mathcal{A}^{\text{ch}, \widehat{H}}_0(\widehat{U}_{\alpha})$. We now define
\begin{equation} \phi^{\text{ch}}_n: \Omega^{\text{ch}, H}_{-n}(U_{\alpha}) \rightarrow \mathcal{A}^{\text{ch}, \widehat{H}}_{n}(\widehat{U}_{\alpha})\end{equation} inductively as follows
\begin{equation} \begin{split} & \phi^{\text{ch}}_n(e^{-n\theta_{\alpha}}) =  s^{n}_{\alpha},
\\ & \phi^{\text{ch}}_n(\nu_{(k)}(e^{-n\theta_{\alpha}})) = (\phi^{\text{ch}}_0(\nu))_{(k)}(s^{n}_{\alpha}), \ \text{for all} \ n, k \in \mathbb{Z}, \ \text{and} \ \nu \in \Omega^{\text{ch}, H}_{0}(U_{\alpha}).\end{split} \end{equation} 
In particular, the $\Omega^{\text{ch}, H}_0(U_{\alpha})$-module structure on $\Omega^{\text{ch}, H}_{-n}(U_{\alpha})$ and the $\mathcal{A}^{\text{ch}, \widehat{H}}_{0}(\widehat{U}_{\alpha})$-module structure on $\mathcal{A}^{\text{ch}, \widehat{H}}_{n}(\widehat{U}_{\alpha})$, are intertwined by $\phi^{\text{ch}}_n$, i.e.,
\begin{equation} \label{modulecompatible} \phi^{\text{ch}}_n(\nu_{(k)} \omega) = (\phi^{\text{ch}}_0(\nu))_{(k)} (\phi^{\text{ch}}_n(\omega)), \ \text{for all} \ n,k \in \mathbb{Z},\ \nu \in \Omega^{\text{ch}, H}_{0}(U_{\alpha}) \ \text{and} \ \omega \in \Omega^{\text{ch}, H}_{-n}(U_{\alpha}). \end{equation}
Note that since we have not assigned $\mathcal{A}^{\text{ch}, \widehat{H}}_0(\widehat{U}_{\alpha})$ a weight grading, we must use the notation $\eta_{(k)}$ rather than $\eta_k$ for $\eta \in \mathcal{A}^{\text{ch}, \widehat{H}}_0(\widehat{U}_{\alpha})$.

We now define the exotic chiral de Rham complex
\begin{equation} \cA^{\text{ch}, \widehat{H}}(\widehat{U}_{\alpha}) = \bigoplus_{n\in \mathbb{Z}} \mathcal{A}^{\text{ch}, \widehat{H}}_n(\widehat{U}_{\alpha}),\end{equation} where as usual this means the Fr\'echet space completion of the usual direct sum. We give $\mathcal{A}^{\text{ch}, \widehat{H}}(\widehat{U_{\alpha}})$ a filtration 
\begin{equation} \label{def:wtfiltration} \mathcal{A}^{\text{ch}, \widehat{H}}(\widehat{U_{\alpha}})_{[0]} \subseteq \mathcal{A}^{\text{ch}, \widehat{H}}(\widehat{U_{\alpha}})_{[1]} \subseteq \mathcal{A}^{\text{ch}, \widehat{H}}(\widehat{U_{\alpha}})_{[2]} \cdots, \qquad \mathcal{A}^{\text{ch}, \widehat{H}}(\widehat{U_{\alpha}}) = \bigcup_{i\geq 0} \mathcal{A}^{\text{ch}, \widehat{H}}(\widehat{U_{\alpha}})_{[i]},\end{equation} which we call the {\it weight filtration}, defined on generators follows:
\begin{equation} \label{def:weight} \begin{split}
 & \text{wt}(f) = \text{wt}(\omega) = \text{wt}(\widehat{A}) = \text{wt} (s^{n}_{\alpha}) = 0,\
\\ & \text{wt}(\iota_X) = \text{wt}(L_X) = \text{wt}(L_A) = \text{wt}(\iota_{\widehat{A}} )= \text{wt}(\Gamma^A) \leq 1.
\end{split} \end{equation} In other words, $f, \omega, \widehat{A}, s^n_{\alpha}$ lie in $\mathcal{A}^{\text{ch}, \widehat{H}}(\widehat{U_{\alpha}})_{[0]}$ and $\iota_X, L_X, L_X, \iota_{\widehat{A}}$ lie in $\mathcal{A}^{\text{ch}, \widehat{H}}(\widehat{U_{\alpha}})_{[1]}$. Elements of $\mathcal{A}^{\text{ch}, \widehat{H}}(\widehat{U_{\alpha}})_{[i]}$ are said to have {\it weight at most $i$}. If $a \in \mathcal{A}^{\text{ch}, \widehat{H}}(\widehat{U_{\alpha}})_{[i]}$, we set $\partial a \in \mathcal{A}^{\text{ch}, \widehat{H}}(\widehat{U_{\alpha}})_{[i+1]}$. It is apparent from the OPE algebra \eqref{ope:exoticcdr} that if $a \in \mathcal{A}^{\text{ch}, \widehat{H}}(\widehat{U_{\alpha}})_{[i]}$ and $b \in \mathcal{A}^{\text{ch}, \widehat{H}}(\widehat{U_{\alpha}})_{[j]}$, then $a_{(k)} b \in \mathcal{A}^{\text{ch}, \widehat{H}}(\widehat{U_{\alpha}})_{[i+j-k-1]}$ for all $i,j\geq 0$. Note that the weight zero component $\mathcal{A}^{\text{ch}, \widehat{H}}_n(\widehat{U_{\alpha}})_{[0]}$ consists of linear combinations of elements of the form $:(\omega + \widehat{A} \nu) s^n_{\alpha}:$, which we can identity with the space of exotic differential forms. In particular, under coordinate transformations the element $s^n_{\alpha} \in \mathcal{A}^{\text{ch}, \widehat{H}}(\widehat{U_{\alpha}})_{[0]}$ transforms as a section of $\widehat{\pi}^* (L^{\otimes n})$.

We assemble the maps $\phi^{\text{ch}}_n$ for all $n \in \mathbb{Z}$ to construct the map
\begin{equation} \label{def:phi} \phi^{\text{ch}}: \Omega^{\text{ch}, H}(U_{\alpha}) \rightarrow \mathcal{A}^{\text{ch}, \widehat{H}}(\widehat{U}_{\alpha}),\end{equation} such that $\phi^{\text{ch}}$ restricts to $\phi^{\text{ch}}_n$ on the summand $\Omega^{\text{ch}, H}_{-n}(U_{\alpha})$. It is straightforward to check using Lemma \ref{lem:degreezeroiso} combined with \eqref{cdr:additional}, \eqref{additional:exotic}, and  \eqref{modulecompatible}, that $\phi^{\text{ch}}$ preserves all OPEs. It is bijective by the same argument as the proof of Lemma \ref{lem:degreezeroiso}. By Remark \ref{rem:cdrlocal}, we obtain

\begin{theorem} The map $\phi^{\rm ch}$ is an isomorphism of vertex algebras for each index $\alpha$. Moreover, it induces an isomorphism of sheaves of vertex algebras over $M$,
\begin{equation} \label{def:phisheaf} \phi^{\rm ch}: \pi_* \big(\Omega^{{\rm ch}, H}_{Z}\big) \rightarrow \widehat{\pi}_* \big(\mathcal{A}^{{\rm ch}, \widehat{H}}_{\widehat{Z}}\big).\end{equation} Taking global sections, we get a vertex algebra isomorphism
\begin{equation} \label{def:phiglobal} \phi^{\rm ch}: \Omega^{{\rm ch}, H}(Z) \rightarrow \mathcal{A}^{{\rm ch}, \widehat{H}}(\widehat{Z}).\end{equation}
\end{theorem}

\begin{remark} The structure of $\Omega^{{\rm ch}, H}_{Z}$ does not depend on our choice of connection form $A$ or flux form $H$, since it is isomorphic to the untwisted chiral de Rham complex $\Omega^{{\rm ch}}_{Z}$. Therefore the structure of $\mathcal{A}^{{\rm ch}, \widehat{H}}_{\widehat{Z}}$ also does not depend on these choices or on the choice of $\widehat{A}$ or $\widehat{H}$, although the isomorphism \eqref{def:phisheaf} does depend on these choices.
\end{remark}

\section{Chiral Han-Mathai map}
Recall that $\Omega^{\text{ch},H}(U_{\alpha})$ has weight grading $\Omega^{\text{ch},H}(U_{\alpha}) = \bigoplus_{n\geq 0} \Omega^{\text{ch},H}(U_{\alpha})[n]$, and hence has the associated weight filtration $$\Omega^{\text{ch},H}(U_{\alpha})_{[0]} \subseteq \Omega^{\text{ch},H}(U_{\alpha})_{[1]} \subseteq \Omega^{\text{ch},H}(U_{\alpha})_{[2]} \subseteq \cdots, \qquad \Omega^{\text{ch},H}(U_{\alpha}) = \bigcup_{n\geq 0} \Omega^{\text{ch},H}(U_{\alpha})_{[n]},$$ where $\Omega^{\text{ch},H}(U_{\alpha})_{[n]} = \bigoplus_{i=0}^n \Omega^{\text{ch},H}(U_{\alpha})[i]$.

We interpret the map $\phi^{\text{ch}}$ as the analogue of the Cavalcanti-Gualtieri isomorphism of Courant algebroids \cite{cavalcanti}. However, it is clear from \eqref{tau0} that $\phi^{\text{ch}}$ does not preserve the weight filtration, and does not have a degree shift, so it is not the chiralization of the Han-Mathai map $\tau: \Omega^{\bar{k}}(Z) \ra \mathcal{A}^{\overline{k + 1}}(\widehat{Z})^{\widehat{\TT}}$. To define the analogue of $\tau$, we need to regard 
$\Omega^{\text{ch},H}_{Z}$ not as a vertex algebra sheaf, but as a sheaf of {\it modules over itself}. For each $U_{\alpha}$, $\Omega^{\text{ch},H}(U_{\alpha})$ is generated by the vacuum vector $1$ as a module over itself. Similarly, we regard $\mathcal{A}^{\text{ch}, \widehat{H}}_{\widehat{Z}}$ not as a sheaf of vertex algebras, but as a sheaf of modules over itself. Both $\Omega^{\text{ch},H}(U_{\alpha})$ and  $\mathcal{A}^{\text{ch},\widehat{H}}(\widehat{U}_{\alpha})$ are $\mathbb{Z}_2$-graded, where the grading is the $\mathbb{Z}_2$-reduction of the degree grading. We shall call this $\mathbb{Z}_2$-grading the {\it degree}, and for $\bar{k} \in \mathbb{Z}_2$, we use the notation 
$$\Omega^{\text{ch},H,\bar{k}}(U_{\alpha}) = \bigoplus_{n\in \mathbb{Z}} \Omega_n^{\text{ch},H,\bar{k}}(U_{\alpha}),\qquad \mathcal{A}^{\text{ch}, \widehat{H},\bar{k}}(\widehat{U}_{\alpha}) = \bigoplus_{n\in \mathbb{Z}} \mathcal{A}_n^{\text{ch}, \widehat{H},\bar{k}}(\widehat{U}_{\alpha})$$ to denote the $\mathbb{Z}_2$-graded components, and similarly for the corresponding sheaves. Note that the map $\phi^{\text{ch}}$ defined in \eqref{def:phi} preserves this grading.

Recall that $\mathcal{A}^{\text{ch},\widehat{H}}(\widehat{U}_{\alpha})$ is only filtered by weight rather than graded, so for $\eta\in \mathcal{A}^{\text{ch},\widehat{H}}(\widehat{U}_{\alpha})$, the vertex algebra operation $\eta_{(k)}$ is well-defined, but $\eta_k$ is not. However, it will be convenient to give meaning to $\eta_k$ in the case when $\eta = \phi^{\text{ch}}(\nu)$ and $\nu \in  \Omega^{\text{ch},H}(U_{\alpha})$ is one of the weight-homogeneous generators 
$$f,\quad \omega,\quad s^n_{\alpha}, \quad \iota_X, \quad L_X-\iota_X H^3 + :A (\iota_X H^{2}):,\quad A, \quad\iota_A, \quad L_A - H^2, \quad \Gamma^A.$$ We define
\begin{equation} \begin{split} \label{phishift} & (\phi^{\text{ch}}(f))_k = f_k = f_{(k-1)}, \qquad (\phi^{\text{ch}}(\omega))_k = \omega_k = \omega_{(k-1)},
\\ & (\phi^{\text{ch}}(e^{-n\theta_{\alpha}}))_k = (s^n_{\alpha})_k = (s^n_{\alpha})_{(k-1)}, \qquad (\phi^{\text{ch}}(\iota_X))_k = (\iota_X)_k = (\iota_X)_{(k)},
\\ & (\phi^{\text{ch}}(L_X-\iota_X H^3 + :A (\iota_X H^{2}):))_k = (L_X -  \iota_X H^3)_k = (L_X -  \iota_X H^3)_{(k)}, 
\\ & (\phi^{\text{ch}}(A))_k = (\iota_{\widehat{A}})_k = (\iota_{\widehat{A}})_{(k)}, \qquad  (\phi^{\text{ch}}(\iota_A))_k = (\widehat{A})_k = (\widehat{A})_{(k-1)},
\\ & (\phi^{\text{ch}}(L_A - H^2))_k = (L_A)_k = (L_A)_{(k)},\qquad (\phi^{\text{ch}}(\Gamma^A))_k = (\Gamma^A)_k =(\Gamma^A)_{(k)} .\end{split} \end{equation} 
For each $U_{\alpha}$, we now define a linear map 
\begin{equation} \label{tchiralmodule} \tau^{\text{ch}}: \Omega^{\text{ch},H,\bar{k}}(U_{\alpha}) \rightarrow  \mathcal{A}^{\text{ch},\widehat{H},\overline{k+1}}(\widehat{U}_{\alpha}),\end{equation} 
inductively as follows:
\begin{equation} \tau^{\text{ch}}(1) = \widehat{A}, \qquad \tau^{\text{ch}}(\nu_k (\mu)) = (-1)^{|\nu|} (\phi^{\text{ch}}(\nu))_k (\tau^{\text{ch}}(\mu)).\end{equation}
Here $\nu$ is one of the weight-homogeneous generators of $\Omega^{\text{ch},H}(U_{\alpha})$ regarded as a vertex algebra, and $\mu$ lies in $\Omega^{\text{ch},H}(U_{\alpha})$ regarded as a $\Omega^{\text{ch},H}(U_{\alpha})$-module. The fact that $\tau^{\text{ch}}$ is well-defined is a consequence of the standard quasi-commutativity and quasi-associativity formulas in vertex algebra theory. We regard $\tau^{\text{ch}}$ not as a vertex algebra homomorphism, but as a homomorphism of vertex algebra modules in the sense that it intertwines that action of $\Omega^{\text{ch},H}(U_{\alpha})$ on itself, and $\mathcal{A}^{\text{ch},\widehat{H}}(\widehat{U}_{\alpha})$ on itself, via the homomorphism $\phi^{\text{ch}}$. We obtain a homomorphism of sheaves of modules on $M$ 
\begin{equation} \label{T-duality-module}\tau^{\text{ch}}: \pi_*\big( \Omega^{\text{ch},H,\bar{k}}\big)_{Z}  \rightarrow  \widehat{\pi}_*\big( \mathcal{A}^{\text{ch},\widehat{H},\overline{k+1}}\big)_{\widehat{Z}},\end{equation}  which we also denote by $\tau^{\text{ch}}$. In particular, we get a homomorphism of modules of global sections
\begin{equation} \label{T-duality-moduleglobal}\tau^{\text{ch}}: \Omega^{\text{ch},H,\bar{k}}(Z) \rightarrow  \mathcal{A}^{\text{ch},\widehat{H},\overline{k+1}}(\widehat{Z}).\end{equation}

\begin{theorem}  The map $\tau^{\rm ch}$ shifts the $\mathbb{Z}_2$-grading and preserves the weight filtration, i.e., $$\tau^{\rm ch}(\Omega^{{\rm ch},H,\bar{k}}(Z)_{[i]}) \subseteq \mathcal{A}^{{\rm ch},\widehat{H},\overline{k+1}}(\widehat{Z})_{[i]}.$$  Moreover, $\tau^{\rm ch}$ coincides at weight zero with the classical T-duality map of Han and Mathai. 
\end{theorem}

\begin{proof} By definition, $\tau^{\text{ch}}$ preserves weight and shifts degree when applied to the vacuum $1$, since $1$ has weight $0$ and degree $\bar{0}$, and $\tau^{\text{ch}}(1) = \widehat{A}$ has weight zero and degree $\bar{1}$. Inductively, suppose that $\mu$ has weight $d$ and degree $\bar{j}$, and that $\tau^{\text{ch}}(\mu)$ has weight at most $d$ and degree $\overline{j+1}$. Then for any homogeneous generator $\nu \in \Omega^{\text{ch},H,\bar{k}}(Z)$ and $r \in \mathbb{Z}$, $\nu_r \mu$ has weight at most $d - r$ and degree $\overline{j+k}$. Since $\phi^{\text{ch}}$ preserves degree, and $\tau^{\text{ch}}(\nu_r \mu) = (-1)^{|\nu|} (\phi^{\text{ch}}(\nu))_r (\tau^{\text{ch}}(\mu))$ has weight at most $d-r$ and degree $\overline{j+k+1}$, it follows that $\tau^{\text{ch}}$ preserves the weight filtration and shifts degree. 

Note that $$\tau^{\text{ch}}(A) = \tau^{\text{ch}}(A_0 (1)) = -(\phi^{\text{ch}}(A))_0 (\tau^{\text{ch}}(1)) = -(\iota_{\widehat{A}})_0(\widehat{A}) = - 1.$$ Since $\phi^{\text{ch}}(\omega) = \omega$ for all $\omega \in \pi^*(\Omega^1(M))$, we conclude that at weight zero, $\tau^{\text{ch}}$ coincides with Han-Mathai map $\tau$.
\end{proof}

It follows from the definition of the maps \eqref{def:phi} and \eqref{tchiralmodule} that $\tau^{\text{ch}}(e^{-n\theta_{\alpha}}) = \ :s^{n}_{\alpha} \widehat{A}:$, for all $n \neq 0$. Therefore $\tau^{\text{ch}}$ maps $\Omega^{\text{ch},H,\bar{k}}_{-n}(Z)$ to $\cA^{\text{ch},\widehat{H},\overline{k+1}}_n(\widehat{Z})$. We interpret this as exchange of momentum and winding number as in the setting of \cite{HM18}.

Next, we shall define the chiral analogue of the map $\widehat{\sigma}: \cA^{\bar{k}}(\widehat{Z}) \rightarrow \Omega^{\overline{k+1}}(\widehat{Z})$. First, let
$$\widehat{\psi}^{\text{ch}}: \cA^{{\rm ch}, \widehat{H}}(\widehat{Z}) \rightarrow \Omega^{\text{ch}, H}(Z)$$ be the inverse of the vertex algebra isomorphism $\phi^{\text{ch}}$ given by \eqref{def:phi}. We define $$\widehat{\sigma}^{\text{ch}}: \cA^{\text{ch}, \widehat{H},\bar{k}}(\widehat{Z}) \rightarrow \Omega^{\text{ch},\overline{k+1}}(\widehat{Z})$$ inductively as follows:
$$\widehat{\sigma}^{\text{ch}}(1) = A, \qquad \widehat{\sigma}^{\text{ch}}(\nu_k (\mu)) = (-1)^{|\nu|} (\widehat{\psi}^{\text{ch}}(\nu))_k (\widehat{\sigma}^{\text{ch}}(\mu)).$$ Here $\nu$ is one of the generators of $\cA^{\text{ch},\widehat{H}}(\widehat{Z})$ regarded as a vertex algebra, which is the image under $\phi^{\text{ch}}$ of a weight-homogeneous generator of $\Omega^{\text{ch}, H}(Z)$;  namely, $\nu$ is either $f$, $\omega$, $s^n_{\alpha}$ (in local coordinates), $\iota_X$, $L_X - \iota_X H^3$, $L_A$, or $\Gamma^A$. Similarly, $\mu$ lies in $\cA^{\text{ch},\widehat{H}}(\widehat{Z})$ regarded as a module over $\cA^{\text{ch},\widehat{H}}(\widehat{Z})$. Reversing the roles of $Z$ and $\widehat{Z}$, we have the analogous maps 
$$\widehat{\tau}^{\text{ch}}: \Omega^{\text{ch}, \widehat{H},\bar{k}}(\widehat{Z}) \rightarrow \cA^{\text{ch}, H,\overline{k+1}}(Z),\qquad \sigma^{\text{ch}}: \cA^{\text{ch}, H,\bar{k}}(Z) \rightarrow \Omega^{\text{ch},\widehat{H},\overline{k+1}}(\widehat{Z}).$$

\begin{theorem} We have the following identities.
\begin{equation}
\begin{split} & -{\rm Id} = \widehat{\sigma}^{\rm ch} \circ \tau^{\rm ch}: \Omega^{{\rm ch}, H,\bar{k}}(Z) \rightarrow \Omega^{{\rm ch}, H,\bar{k}}(Z),
\\ & -{\rm Id} = \widehat{\tau}^{\rm ch} \circ \sigma^{\rm ch}: \cA^{{\rm ch}, H,\bar{k}}(Z) \rightarrow \cA^{{\rm ch}, H,\bar{k}}(Z).
\end{split}
\end{equation} 
In particular, $\tau^{\rm ch}$ is a linear isomorphism. \end{theorem}

\begin{proof} We only prove the first identity, since the proof of the second one is the same. First, it is clear that it holds on the vacuum vector $1$ since $$\widehat{\sigma}^{\text{ch}} \circ \tau^{\text{ch}}(1) = \widehat{\sigma}^{\text{ch}}(\widehat{A}) = \widehat{\sigma}^{\text{ch}}(\widehat{A}_0 1) = - (\widehat{\psi}^{\text{ch}}(\widehat{A})_0 (\widehat{\sigma}(1)) = -( \iota_A)_0(A)= -1.$$
Next, it suffices to show that if $\widehat{\sigma}^{\text{ch}} \circ \tau^{\text{ch}}(\mu)= - \mu$, then for each weight-homogeneous generator $\nu$ of $\Omega^{\text{ch}, H,\bar{k}}(Z)$, we have $$ \widehat{\sigma}^{\text{ch}} \circ \tau^{\text{ch}}(\nu_k(\mu)) = - (\nu_k(\mu)).$$ To check this, we compute
\begin{equation}
\begin{split} \widehat{\sigma}^{\text{ch}} \circ \tau^{\text{ch}}(\nu_k(\mu)) & = (-1)^{|\nu|} \widehat{\sigma}^{\text{ch}} \bigg( (\phi^{\text{ch}}(\nu))_k \tau^{\text{ch}}(\mu)\bigg)
\\ &  = (-1)^{|\nu|}  (-1)^{|\phi^{\text{ch}}(\nu)|} \ \big(\widehat{\psi}^{\text{ch}}(\phi^{\text{ch}}(\nu))\big)_k \big(\widehat{\sigma}^{\text{ch}} (\tau^{\text{ch}}(\mu))\big)
\\ & =\nu_k (-\mu) = - \nu_k(\mu),
\end{split}
\end{equation} 
since $\widehat{\psi}^{\text{ch}}(\phi^{\text{ch}}(\nu)) = \nu$ and $\widehat{\sigma}^{\text{ch}} (\tau^{\text{ch}}(\mu)) = -\mu$.
\end{proof}

\section{Differential structure on $\mathcal{A}^{{\rm ch}, \widehat{H},\bar{k}}(\widehat{Z})$}

The final step is to equip $\mathcal{A}^{\text{ch}, \widehat{H}}(\widehat{Z})$ with a square-zero twisted differential $D_{\widehat{Z}, \widehat{H}}$ with the following properties.

\begin{enumerate}
\item $D_{\widehat{Z}, \widehat{H}}$ shifts the $\mathbb{Z}_2$-graded degree and preserves the weight filtration, that is, $$D_{\widehat{Z}, \widehat{H}} (\mathcal{A}^{\text{ch}, \widehat{H},\bar{k}}(\widehat{Z})_{[m]}) \subseteq \mathcal{A}^{\text{ch}, \widehat{H},\overline{k+1}}(\widehat{Z})_{[m]}.$$

\item On $\mathcal{A}^{\text{ch}, \widehat{H}}(\widehat{Z})_{[0]}$, $D_{\widehat{Z}, \widehat{H}}$ restricts to the exotic differential  $\nabla^{L^{\otimes n}} - n \iota_{\widehat{A}} + \widehat{H}$. In particular, the weight zero subcomplex $(\mathcal{A}^{\text{ch}, \widehat{H}}(\widehat{Z})_{[0]}, D_{\widehat{Z}, \widehat{H}})$ coincides with the exotic complex of Han and Mathai.

\item At weight zero, $\tau^{\text{ch}}$ intertwines the twisted differentials up to a sign, that is, \begin{equation} \label{intertwiningofd} \tau^{\text{ch}}\circ D_H = -D_{\widehat{Z}, \widehat{H}} \circ \tau^{\text{ch}}. \end{equation} In this notation, $D_H$ is the twisted differential on $\Omega^{\text{ch},H,\bar{k}}(Z)$ given by $D_H (\nu) = D(\nu) + :H \nu:$, where $D$ is the chiral de Rham differential.
\end{enumerate}

By Theorem 2 of \cite{LM15}, the cohomology of $(\Omega^{\text{ch},H}(Z), D_H)$ vanishes in positive weight, and coincides with the classical twisted cohomology in weight zero. Unfortunately, the intertwining property \eqref{intertwiningofd} no longer holds in positive weight, so it is not obvious whether the cohomology of $(\mathcal{A}^{\text{ch}, \widehat{H}}(\widehat{Z}), D_{\widehat{Z}, \widehat{H}})$ vanishes in positive weight. We expect that for all $m$, the inclusions of complexes \begin{equation} (\mathcal{A}^{\text{ch}, \widehat{H}}(\widehat{Z})_{[0]}, D_{\widehat{Z}, \widehat{H}}) \hookrightarrow  (\mathcal{A}^{\text{ch}, \widehat{H}}(\widehat{Z})_{[m]}, D_{\widehat{Z}, \widehat{H}}) \hookrightarrow  (\mathcal{A}^{\text{ch}, \widehat{H}}(\widehat{Z}), D_{\widehat{Z}, \widehat{H}})\end{equation} are all quasi-isomorphisms, that is, they induce isomorphisms in cohomology. In the last section, we will specialize to the case where both circle bundles $Z$ and $\widehat{Z}$ are trivial, and the fluxes $H$ and $\widehat{H}$ are both zero, and we will prove that this statement holds in this case.

We shall define $D_{\widehat{Z}, \widehat{H}}$ in two steps. Recall first that the chiral de Rham differential $D$ on $\Omega^{\text{ch},H}(Z)= \bigoplus_{n\in \mathbb{Z}} \Omega^{\text{ch}, H}_n(Z)$ is a vertex algebra derivation given on generators by
\begin{equation} \begin{split} & D(f) = df, \quad D(\omega) = d\omega, \quad D(\iota_X) = L_X - \iota_X H = L_X - \iota_X H^{3} +  :A ( \iota_X H^{2}):,
\\ & D( L_X) = L_X H^3 + :H^2 (\iota_X \widehat{H}^2): + :A (L_X H^{2}):,
\\ &  D(\iota_A) = L_A - \iota_A H = L_A - H^{2}, \quad D(A) = dA = \widehat{H}^{2},
\\ &  D(\Gamma^A) =  \partial A -\xi^A :,
\\ & D(e^{n\theta_{\alpha}}) = n :e^{n \theta_{\alpha}} d\theta_{\alpha}:\   =  n :( A_{\alpha} - A_{\alpha,\text{bas}}) e^{n \theta_{\alpha}}: .\end{split} \end{equation}
As in \cite{LM15}, $\xi^A$ has degree $1$ and weight $1$ and satisfies $D \xi^A = \partial dA$. We can transport this structure to $ \mathcal{A}^{\text{ch}, \widehat{H}}(\widehat{Z}) = \bigoplus_{n\in \mathbb{Z}} \mathcal{A}^{\text{ch}, \widehat{H}}_n(\widehat{Z})$ by defining the differential $D_{\widehat{Z}}$ on generators as follows 

\begin{equation} \label{actiondtildez} \begin{split} & D_{\widehat{Z}}(f) = d f, \quad  D_{\widehat{Z}}(\omega) = d \omega,\quad D_{\widehat{Z}}(\iota_X) = L_X- \iota_X H^{3}, 
\\ & D_{\widehat{Z}}(L_X) = L_X H^3 + :H^2 (\iota_X \widehat{H}^2): + :\widehat{H}^2 (\iota_X H^2): ,\quad D_{\widehat{Z}}(\iota_{\widehat{A}}) = \widehat{H}^{2} ,
\\ & D_{\widehat{Z}}(\widehat{A}) = L_A  , \quad D_{\widehat{Z}} (L_A) = 0,\quad D_{\widehat{Z}}(\Gamma^A)  = \partial \iota_{\widehat{A}} - \xi^A, 
\\ & D_{\widehat{Z}}(s^n_{\alpha}) = -n :\iota_{\widehat{A}} s^n_{\alpha}: +n :A_{\alpha, \text{bas}} s^n_{\alpha}:. \end{split} \end{equation} By construction, we have $$\phi^{\text{ch}} \circ D = D_{\widehat{Z}} \circ \phi^{\text{ch}}.$$ In other words, $\phi^{\text{ch}}$ is an isomorphism of differential vertex algebras. In particular, there exists a locally defined field $D_{\widehat{Z}}(z)$ whose zero-mode is globally well-defined and coincides with $D_{\widehat{Z}}$. Therefore $D_{\widehat{Z}}$ is a square-zero derivation on the algebra. It is clearly homogeneous of degree $\bar{1}$, that is 
$$D_{\widehat{Z}}(\mathcal{A}^{\text{ch}, \widehat{H},\bar{k}}(\widehat{Z})) \subseteq \mathcal{A}^{\text{ch}, \widehat{H},\overline{k+1}}(\widehat{Z}).$$ We caution the reader that neither $\phi^{\text{ch}}$ nor $D_{\widehat{Z}}$ preserve the weight filtration. Next, we modify $D_{\widehat{Z}}$ as follows. We define
\begin{equation} \begin{split} D_{\widehat{Z}, \widehat{H}}  & =  D_{\widehat{Z}} + D^0+ D^1 + D^2 + D^3 +D^4+D^5+D^6,
\\ & D^0 = - (:\widehat{A} \widehat{H}^{2}:)_{(0)},\qquad D^1 =   (: H^{2}\iota_{\widehat{A}}:)_{(0)},\qquad D^2 = -  (:\iota_{\widehat{A}} L_A:)_{(0)}, \qquad D^3 =   H^3_{(0)},
\\ & D^4 =  (:\iota_{\widehat{A}} L_A:)_{(1)}, \qquad D^5 =   (: H^{2}\iota_{\widehat{A}}:)_{(1)}, \qquad D^6 = \widehat{H}_{(-1)} = H^3_{(-1)} + (:\widehat{A} \widehat{H}^{2}:)_{(-1)}.\end{split} \end{equation}
We observe first that $D_{\widehat{Z}, \widehat{H}}$ is well-defined globally and homogeneous of degree $\bar{1}$. Note that $D_{\widehat{Z}} + D^0+D^1+D^2+D^3$ is a vertex algebra derivation, since $D_{\widehat{Z}}$, as well as the zero-mode of any field, has this property. The terms $D^4$ and $D^5$, being first modes of fields, are not derivations. We will need the following computations repeatedly for the remainder of this section.
\begin{equation} \label{actiond0} \begin{split}
& D^0(f) = 0,\qquad  D^0(\omega) = 0, \qquad D^0(s^n_{\alpha}) = 0,
\\ &  D^0(\iota_X) = \ :\widehat{A} (\iota_X \widehat{H}^2):,\qquad D^0(L_X) = \ :\widehat{A} (L_X \widehat{H}^2):,
\\ & D^0(\widehat{A}) = 0,\qquad D^0(\iota_{\widehat{A}}) = - \widehat{H}^2, \qquad D^0(L_A) = 0, \qquad D^0(\Gamma^A) = 0.
\end{split}\end{equation}

\begin{equation} \label{actiond1} \begin{split}
& D^1(f) = 0,\qquad  D^1(\omega) = 0, \qquad D^1(s^n_{\alpha}) = 0,
\\ & D^1(\iota_X) =   - :\iota_{\widehat{A}} (\iota_X H^2):,\qquad D^1(L_X) =   -  :\iota_{\widehat{A}} (L_X H^2): - :H^2 (\iota_X \widehat{H}^2):,
\\ & D^1(\widehat{A}) = H^2 \qquad D^1(\iota_{\widehat{A}}) = 0, \qquad D^1(L_A) = 0, \qquad D^1(\Gamma^A )= 0,
\end{split}\end{equation}

\begin{equation} \label{actiond2} \begin{split}
& D^2(f) = 0,\qquad  D^2(\omega) = 0, \qquad D^2(s^n_{\alpha}) = n : \iota_{\widehat{A}} s^n_{\alpha}:,
\\ & D^2(\iota_X) =  0,\qquad D^2(L_X) =\  :L_A (\iota_X \widehat{H}^2):,
\\ & D^2(\widehat{A}) =  - L_A,\qquad D^2(\iota_{\widehat{A}}) = 0, \qquad D^2(L_A) = 0, \qquad D^2(\Gamma^A )= -\partial \iota_{\widehat{A}},
\end{split}\end{equation}

\begin{equation} \label{actiond3} \begin{split}
& D^3(f) = 0,\qquad  D^3(\omega) = 0, \qquad D^3(s^n_{\alpha}) = 0,
\\ & D^3(\iota_X) =  \iota_X H^3 ,\qquad D^3(L_X) =  -L_X H^3 ,
\\ & D^3(\widehat{A}) = 0,\qquad D^3(\iota_{\widehat{A}}) = 0, \qquad D^3(L_A) = 0, \qquad D^3(\Gamma^A )= 0,
\end{split}\end{equation}

\begin{equation} \label{actiond4} \begin{split}
& D^4(f) = 0,\qquad  D^4(\omega) = 0, \qquad D^4(s^n_{\alpha}) = 0,
\\ & D^4(\iota_X) =  0,\qquad D^4(L_X) = 0,
\\ & D^4(\widehat{A}) = 0,\qquad D^4(\iota_{\widehat{A}}) = 0, \qquad D^4(L_A) = 0, \qquad D^4(\Gamma^A )=  \iota_{\widehat{A}},
\end{split}\end{equation}

\begin{equation} \label{actiond5} \begin{split}
& D^5(f) = 0,\qquad  D^5(\omega) = 0, \qquad D^5(s^n_{\alpha}) = 0,
\\ & D^5(\iota_X) =  0,\qquad D^5(L_X) = 0,
\\ & D^5(\widehat{A}) = 0,\qquad D^5(\iota_{\widehat{A}}) = 0, \qquad D^5(L_A) = 0, \qquad D^5(\Gamma^A )= 0,
\end{split}\end{equation}

\begin{equation} \label{actiond6} \begin{split}
& D^6(f) = \ :H^3 f: + :\widehat{A} \widehat{H^2} f:, \quad  D^6(\omega) = \ :H^3 \omega: + :\widehat{A} \widehat{H^2} \omega:, \quad D^6(s^n_{\alpha}) = \ :H^3 s^n_{\alpha}: + :\widehat{A} \widehat{H^2} s^n_{\alpha}:,
\\ & D^6(\iota_X) =   \ :H^3 \iota_X: + :(\widehat{A} \widehat{H^2}) \iota_X: \ = \ :H^3 \iota_X: + :\widehat{A} \widehat{H^2} \iota_X:   - :\partial \widehat{A} (\iota_X \widehat{H^2}):,
\\ & D^6(L_X) =  \ :H^3 L_X: + :(\widehat{A} \widehat{H^2}) L_X:\  =   \ :H^3 L_X: + :\widehat{A} \widehat{H^2} L_X: - :\partial \widehat{A} (L_X\widehat{H^2}):,
\\ & D^6(\widehat{A}) =  \ :H^3 \widehat{A}: ,\qquad D^6(\iota_{\widehat{A}})  =  \ :H^3 \iota_{\widehat{A}}: + :(\widehat{A} \widehat{H^2})\iota_{\widehat{A}}:\  =   \ :H^3 \iota_{\widehat{A}}: + :\widehat{A} \widehat{H^2} \iota_{\widehat{A}}: + \partial \widehat{H^2},
\\ & D^6(L_A) = \ :H^3 L_A: + :\widehat{A} \widehat{H^2} L_A:, \qquad D^6(\Gamma^A )= \ :H^3 \Gamma^A: + :\widehat{A} \widehat{H^2} \Gamma^A:.
\end{split}\end{equation}

\begin{lemma} The operator $D_{\widehat{Z}, \widehat{H}}$ on $\mathcal{A}^{{\rm ch}, \widehat{H}}(\widehat{Z})$ preserves the weight filtration \eqref{def:wtfiltration}. In particular, $D_{\widehat{Z}, \widehat{H}}$ acts on the weight zero subspace $\mathcal{A}^{{\rm ch}, \widehat{H}}(\widehat{Z})_{[0]}$. \end{lemma}

\begin{proof} Note that $D_{\widehat{Z}, \widehat{H}} = D' + D''$ where $D' = D_{\widehat{Z}} + D^2$ and $D'' = D^0 + D^1 + D^3 +D^4 +D^5 +D^6$. Since $D'$ is a vertex algebra derivation, to show that it preserves the weight filtration it suffices to check this on generators, and this is apparent from \eqref{actiondtildez} and \eqref{actiond2}. Even though $D''$ is not a derivation, it is apparent from \eqref{def:weight} that $D''$ consists of terms which either preserve or lower the weight. This completes the proof. \end{proof}

\begin{lemma} The weight zero subcomplex $\big(\mathcal{A}^{{\rm ch}, \widehat{H}}(\widehat{Z})_{[0]}, D_{\widehat{Z}, \widehat{H}}\big)$ can be identified with the Han-Mathai complex. In particular, on $\mathcal{A}^{{\rm ch}, \widehat{H}}(\widehat{Z})_{[0]}$ we have  $\tau^{\rm ch}\circ D_H = -D_{\widehat{Z}, \widehat{H}} \circ \tau^{\rm ch}$.
\end{lemma}

\begin{proof} In local coordinates, the general element of $g \in \mathcal{A}^{\text{ch}, \widehat{H}}_n(\widehat{Z})_{[0]}$ has the form $$g =\ :\omega_0 s^n_{\alpha}: + :\widehat{A} \omega_1 s^n_{\alpha}:$$ where $\omega_0, \omega_1$ are basic differential forms. We compute
\begin{equation} \begin{split}  & D_{\widehat{Z}, \widehat{H}} \bigg(: \omega_0s^n_{\alpha}: + :\widehat{A} \omega_1s^n_{\alpha}:\bigg) 
\\ & = \ :(d \omega_0)s^n_{\alpha} : + (-1)^{|\omega_0|} n :\omega_0 A_{\alpha, \text{bas}} s^n_{\alpha}: + :H^2 \omega_1 s^n_{\alpha}: - :\widehat{A}(d\omega_1) s^n_{\alpha} : - (-1)^{|\omega_1|}n :\widehat{A} \omega_1  A_{\alpha,\text{bas}} s^n_{\alpha} :
\\ &-  n \omega_1 s^n_{\alpha} + :H^3 \omega_0 s^n_{\alpha}: + : H^3\widehat{A} \omega_1 s^n_{\alpha}: + :\widehat{A} \widehat{H}^2 \omega_0 s^n_{\alpha}:.
\end{split} \end{equation}

On the other hand, identifying $g$ with the element $ \omega_0 \wedge s^n_{\alpha} + \widehat{A} \wedge \omega_1 \wedge s^n_{\alpha}$ of the Han-Mathai complex, it is apparent from \eqref{nabla} that $D_{\widehat{Z}, \widehat{H}}$ corresponds to $(\widehat{\pi}^* \nabla^{L^{\otimes n}}- \iota_{ n \widehat{\nu}} + \widehat{H})(g)$. In this notation, the operator $\iota_{n\widehat{\nu}}$ is identified with $\iota_{n \widehat{A}} = n \iota_{\widehat{A}}$. The statement that $\tau^{\text{ch}}$ intertwines the differentials $D_H$ and $D_{\widehat{Z}, \widehat{H}}$ up to sign is a straightforward computation.
\end{proof}

Our main result in this section is the following
\begin{theorem} \label{sec5main} $D_{\widehat{Z}, \widehat{H}}$ is a square-zero operator on $\mathcal{A}^{{\rm ch}, \widehat{H}}(\widehat{Z})$.
\end{theorem}

The proof is quite involved, and it depends crucially on the nonassociativity of the normally ordered product, and the fact that $D_{\widehat{Z}, \widehat{H}}$ fails to be a derivation in the category of vertex algebra modules due to the terms $D^4$ and $D^5$. In order to prove Theorem \ref{sec5main}, we observe that $\mathcal{A}^{\text{ch}, \widehat{H}}(\widehat{Z})$ has the following sequence of vertex subalgebras which are all closed under the action of $D_{\widehat{Z}, \widehat{H}}$:
\begin{equation} \langle \Omega (M) \rangle \subseteq \mathcal{A}_0^{\text{ch}, \widehat{H}}(\widehat{Z})^{i \mathbb{R}[t]} \subseteq \mathcal{A}_0^{\text{ch}, \widehat{H}}(\widehat{Z}) \subseteq \mathcal{A}^{\text{ch}, \widehat{H}}(\widehat{Z}). \end{equation}

In this notation, 
\begin{enumerate} 
\item $\langle \Omega(M) \rangle$ denotes the abelian vertex algebra generated by all differential forms on $M$, 
\item $\mathcal{A}_0^{\text{ch}, \widehat{H}}(\widehat{Z})^{i \mathbb{R}[t]}$ denotes the ${i \mathbb{R}[t]}$-invariant subalgebra of $\mathcal{A}_0^{\text{ch}, \widehat{H}}(\widehat{Z})$, which is generated by $\Omega(M)$ together with $\iota_X, L_X, \widehat{A}, \iota_{\widehat{A}}, L_A$,
\item $\mathcal{A}_0^{\text{ch}, \widehat{H}}(\widehat{Z})$ is generated by the above fields together with $\Gamma^A$,
\item $\mathcal{A}^{\text{ch}, \widehat{H}}(\widehat{Z})$ is generated by the above fields together with $s^n_{\alpha}$ in local coordinates, for all $n\in \mathbb{Z}$.
\end{enumerate}

We will proceed by proving Theorem \ref{sec5main} successively on each of these subalgebras, and we organize this as a sequence of lemmas.

\begin{lemma}  $D_{\widehat{Z}, \widehat{H}}$ is a square-zero operator on the subalgebra $\langle \Omega(M) \rangle$.
\end{lemma}

\begin{proof} First, $D^0$, $D^1$, $D^2$, $D^3$, $D^4$, and $D^5$ vanish on $\langle \Omega(M) \rangle$, so $D_{\widehat{Z}, \widehat{H}} = D_{\widehat{Z}} + D^6$. Moreover, $D_{\widehat{Z}}$ is a vertex algebra derivation on $\langle \Omega(M) \rangle$ and $D_{\widehat{Z}, \widehat{H}}$ is a derivation on $\langle \Omega(M) \rangle$ in the category of modules over $\langle \Omega(M) \rangle$. In other words, for all $a,b \in \langle \Omega(M) \rangle$ and $k \in \mathbb{Z}$, we have \begin{equation} \label{modderivation} 
\begin{split} &  D_{\widehat{Z}}(a_{(k)} b) = (D_{\widehat{Z}}(a))_{(k)} b + (-1)^{|a|} a_{(k)} D_{\widehat{Z}}(b),
\\ & D_{\widehat{Z}, \widehat{H}}(a_{(k)} b) = D_{\widehat{Z}}(a)_{(k)} b + (-1)^{|a|} a_{(k)} D_{\widehat{Z}, \widehat{H}}(b). \end{split} \end{equation} 
Since $D_{\widehat{Z}}$ and $D^6$ are commuting differentials which are both square-zero, the claim follows.
\end{proof}

\begin{lemma} $D_{\widehat{Z}, \widehat{H}}$ is a square-zero operator on the subalgebra $\mathcal{A}^{{\rm ch}, \widehat{H}}_0(\widehat{Z})^{i \mathbb{R}[t]}$.
\end{lemma}

\begin{proof} This argument is more difficult than the proof of the previous lemma because the terms $D^4$ and $D^5$ fail to be vertex algebra derivations. We define \begin{equation} \begin{split} & D_{\text{Der}} = D_{\widehat{Z}} + D^0 + D^1 + D^2 +D^3,
\\ & D_{\text{NDer}} = D^4+ D^5,
\\ & D_{\widehat{H}} = D^6. \end{split} \end{equation}
In this notation, $D_{\text{Der}}$ is a vertex algebra derivation, $D_{\text{NDer}}$ is not a derivation, and 
$D_{\widehat{Z}, \widehat{H}} =  D_{\text{Der}}  + D_{\text{NDer}}  + D_{\widehat{H}}$. Then for all $\nu \in \mathcal{A}^{\text{ch}, \widehat{H}}_0(\widehat{Z})^{i \mathbb{R}[t]}$, 
\begin{equation} \begin{split} (D_{\widehat{Z}, \widehat{H}})^2  (\nu) & = (D_{\text{Der}})^2(\nu) +  (D_{\text{NDer}})^2(\nu) +  (D_{\widehat{H}})^2(\nu) 
\\ & + (D_{\text{Der}}  D_{\text{NDer}} + D_{\text{NDer}}  D_{\text{Der}})(\nu) 
\\ & +  (D_{\text{Der}}  D_{\widehat{H}} + D_{\widehat{H}}  D_{\text{Der}})(\nu) 
\\ & +  (D_{\text{NDer}}  D_{\widehat{H}} + D_{\widehat{H}}  D_{\text{NDer}})(\nu)
\\ & =  (D_{\text{Der}})^2(\nu) +  (D_{\text{NDer}}  D_{\widehat{H}} + D_{\widehat{H}}  D_{\text{NDer}})(\nu).
\end{split} \end{equation}
Here were are using the fact that $D_{\text{Der}}$ is a vertex algebra derivation which annihilates the fields $H^3 + :\widehat{A} \widehat{H}^2:$, $:\iota_{\widehat{A}} H^2:$ and $:\iota_{\widehat{A}} L_A:$, so that  $(D_{\text{Der}}  D_{\text{NDer}} + D_{\text{NDer}}  D_{\text{Der}})(\nu) = 0$ and $ (D_{\text{Der}}  D_{\widehat{H}} + D_{\widehat{H}}  D_{\text{Der}})(\nu) = 0$. Also, it is apparent that $ (D_{\text{NDer}})^2(\nu)=0$ and $(D_{\widehat{H}})^2(\nu)$.

Next, we check that $(D_{\widehat{Z}, \widehat{H}})^2$ annihilates the additional generators $\widehat{A}, \iota_{\widehat{A}}, \iota_X, L_X$ that appear in  $\mathcal{A}^{\text{ch}, \widehat{H}}_0(\widehat{Z})^{i \mathbb{R}[t]}$ but not in $\langle \Omega(M) \rangle$. This follows from the following computations.
\begin{equation}\begin{split} & (D_{\text{Der}})^2(\widehat{A})  = 0,\qquad  (D_{\text{NDer}}  D_{\widehat{H}} + D_{\widehat{H}}  D_{\text{NDer}})(\widehat{A}) = 0,
\\  & (D_{\text{Der}})^2(\iota_{\widehat{A}})  = 0,\qquad  (D_{\text{NDer}}  D_{\widehat{H}} + D_{\widehat{H}}  D_{\text{NDer}})(\iota_{\widehat{A}}) = 0,
\\ & (D_{\text{Der}})^2(\iota_X)  = \ :L_A( \iota_X \widehat{H}^2): + :H^2 (\iota_X \widehat{H}^2): + \widehat{H}^2 (\iota_X H^2):,
\\ &  (D_{\text{NDer}}  D_{\widehat{H}} + D_{\widehat{H}}  D_{\text{NDer}})(\iota_X) = - :L_A( \iota_X \widehat{H}^2): - :H^2 (\iota_X \widehat{H}^2): - :\widehat{H}^2 (\iota_X H^2):,
\\ & (D_{\text{Der}})^2(L_X)  = \ :L_A( L_X \widehat{H}^2): + :H^2 (L_X \widehat{H}^2): + :\widehat{H}^2 (L_X H^2):,
\\ &  (D_{\text{NDer}}  D_{\widehat{H}} + D_{\widehat{H}}  D_{\text{NDer}})(L_X) = - :L_A( L_X \widehat{H}^2): - :H^2 (L_X \widehat{H}^2): - :\widehat{H}^2 (L_X H^2):.
\end{split} \end{equation}

Next, a general element of $\mathcal{A}^{\text{ch}, \widehat{H}}_0(\widehat{Z})^{i \mathbb{R}[t]}$ can be expressed as a finite sum of terms of the form
$$ \nu =\ :(\partial^{i_1} \mu_1) \cdots (\partial^{i_r} \mu_r) \eta: , \qquad i_1,\dots, i_r \geq 0,\qquad \eta \in \langle \Omega(M) \rangle,$$ where each $\mu_i$ is one of the generators $\iota_X, L_X, \iota_{\widehat{A}}, \widehat{A}, L_A$.  We say that such a monomial has {\it length} $r$. By the previous lemma, $(D_{\widehat{Z}, \widehat{H}})^2(\nu) = 0$ whenever $\nu$ has length $0$. Inductively, we assume that  $(D_{\widehat{Z}, \widehat{H}})^2(\nu) =0$ whenever $\nu$ is such a monomial of length at most $r-1$. In particular, this means that 
\begin{equation} \label{inductivecond}(D_{\text{Der}})^2(\nu) +  (D_{\text{NDer}}  D_{\widehat{H}} + D_{\widehat{H}}  D_{\text{NDer}})(\nu) = 0.\end{equation}
Now let $\nu = \ :(\partial^{i_1} \mu_1) \cdots (\partial^{i_r} \mu) \eta:$ be a monomial of length $r$ as above, and write 
$$\nu =  \ :(\partial^{i_1} \mu_1) \nu':,\qquad \nu' = \ :(\partial^{i_2} \mu_2) \cdots (\partial^{i_r} \mu_r) \eta:.$$ 
Then
\begin{equation} \begin{split} \label{irtfirst} (D_{\widehat{Z}, \widehat{H}})^2 (\nu)  =  & (D_{\widehat{Z}, \widehat{H}})^2 (:(\partial^{i_1} \mu_1) \nu':)  
\\  = & \ :(D_{\text{Der}})^2 (\partial^{i_1} \mu_1) \nu':+ :(\partial^{i_1} \mu_1) (D_{\text{Der}})^2(\nu'): 
\\ &  +  (D_{\text{NDer}}  D_{\widehat{H}} + D_{\widehat{H}}  D_{\text{NDer}})(:(\partial^{i_1} \mu_1) \nu':))
\\  = &  \ :(D_{\text{Der}})^2 (\partial^{i_1} \mu_1) \nu':+ :(\partial^{i_1} \mu_1) (D_{\text{Der}})^2(\nu'): 
\\ & +  (D_{\text{NDer}}  D_{\widehat{H}} + D_{\widehat{H}}  D_{\text{NDer}})(:(\partial^{i_1} \mu_1) \nu':))
\\ & - :(\partial^{i_1} \mu_1) (D_{\text{NDer}}  D_{\widehat{H}} + D_{\widehat{H}}  D_{\text{NDer}})(\nu'): 
\\ & + :(\partial^{i_1} \mu_1) (D_{\text{NDer}}  D_{\widehat{H}} + D_{\widehat{H}}  D_{\text{NDer}})(\nu'):
\\  = &  \ :(D_{\text{Der}})^2 (\partial^{i_1} \mu_1) \nu':
\\ & +  (D_{\text{NDer}}  D_{\widehat{H}} + D_{\widehat{H}}  D_{\text{NDer}})(:(\partial^{i_1} \mu_1) \nu':))
\\ & - :(\partial^{i_1} \mu_1) (D_{\text{NDer}}  D_{\widehat{H}} + D_{\widehat{H}}  D_{\text{NDer}})(\nu'): 
.\end{split} \end{equation}
The last equality follows from our inductive assumption \eqref{inductivecond} in the case $\nu = \nu'$.

A separate calculation in each of the cases $\mu_1 =\iota_X, L_X, \iota_{\widehat{A}}, \widehat{A}, L_A$ shows that in all cases,
$(D_{\widehat{Z}, \widehat{H}})^2 (\nu)  = 0$. To illustrate this, we include the calculation in the case where $i_1 = 0$ and $\mu = \iota_X$. We have
\begin{equation} \label{irtsecond} \begin{split}  :((D_{\text{Der}})^2(\iota_X)) \nu': \  =  &  \  :(:H^2 (\iota_X \widehat{H}^2):) \nu': + :(:\widehat{H}^2 (\iota_X H^2):)\nu':,
\\ & \ + :( :L_A( \iota_X \widehat{H}^2):) \nu':
\\  (D_{\text{NDer}}  D_{\widehat{H}} + D_{\widehat{H}}  D_{\text{NDer}})(: \iota_X \nu':)  = & \  (:H^2 \widehat{H}^2:)_{(0)} (:\iota_X \nu':) 
+ (:L_A \widehat{H}^2:)_{(0)} (:\iota_X \nu':) 
\\  = & \  -  :(:H^2 (\iota_X \widehat{H}^2):) \nu': - :(:\widehat{H}^2 (\iota_X H^2):)\nu':
\\ & \ -  :( :L_A( \iota_X \widehat{H}^2):) \nu':
\\ & \ +  :\iota_X ((:H^2 \widehat{H}^2:)_{(0)} \nu'): + :\iota_X ((:L_A \widehat{H}^2:)_{(0)} \nu'):,
\\ \  - :\iota_X (D_{\text{NDer}}  D_{\widehat{H}} + D_{\widehat{H}}  D_{\text{NDer}})(\nu'): \  =  &\  - :\iota_X ((H^2 \widehat{H}^2:)_{(0)} \nu'): - :\iota_X ((L_A \widehat{H}^2:)_{(0)} \nu'):
\end{split}
\end{equation} The fact that $(D_{\widehat{Z}, \widehat{H}})^2 (: \iota_X \nu':)  = 0$ then follows immediately from \eqref{irtfirst} and \eqref{irtsecond}. The proof for the other cases is similar and is omitted. \end{proof}

\begin{lemma} \label{invariantsector} $D_{\widehat{Z}, \widehat{H}}$ is a square-zero operator on $\mathcal{A}^{{\rm ch}, \widehat{H}}_0(\widehat{Z})$.
\end{lemma}

\begin{proof} Recall that $\mathcal{A}^{\text{ch}, \widehat{H}}_0(\widehat{Z})$ has one additional generator $\Gamma^A$ in addition to the generators of $\mathcal{A}^{\text{ch}, \widehat{H}}_0(\widehat{Z})^{i \mathbb{R}[t]}$. First, we compute

\begin{equation} \begin{split}  (D_{\text{Der}})^2 (\Gamma^A) = & \ -\partial \widehat{H}^2,
\\ (D_{\text{NDer}}  D_{\widehat{H}} + D_{\widehat{H}}  D_{\text{NDer}})(\Gamma^A) & = \partial \widehat{H}^2. \end{split} \end{equation}
 It is immediate that $(D_{\widehat{Z}, \widehat{H}})^2(\Gamma^A) = 0$. An essential feature of this calculation is the fact that 
$$ :(:\widehat{A} \widehat{H}^2:) \iota_{\widehat{A}}: \ =\  :\widehat{A} \widehat{H}^2 \iota_{\widehat{A}}: + \partial \widehat{H}^2$$ which is due to the nonassociativity of the normally ordered product. Similarly, one checks easily that for all $i \geq 0$,
\begin{equation} \begin{split} & (D_{\text{Der}})^2  (\partial^i \Gamma^A) = -\partial^{i+1}\widehat{H}^2, \qquad  (D_{\text{NDer}}  D_{\widehat{H}} + D_{\widehat{H}}  D_{\text{NDer}})(\partial^i \Gamma^A) = \partial^{i+1}\widehat{H}^2.\end{split} \end{equation}
Next, a general element of $\mathcal{A}^{\text{ch}, \widehat{H},\bar{k}}_0(\widehat{Z})$ can be expressed as a finite sum of terms of the form
$$ \nu =\ :(\partial^{i_1} \Gamma^A) \cdots (\partial^{i_r} \Gamma^A) \eta:, \qquad i_1,\dots, i_r \geq 0,\qquad \eta \in \mathcal{A}^{\text{ch}, \widehat{H},\bar{k}}_0(\widehat{Z})^{i\mathbb{R}[t]}.$$ We say that such a monomial $\nu$ has {\it length} $r$. By the previous lemma, $(D_{\widehat{Z}, \widehat{H}})^2(\nu) = 0$ whenever $\nu$ has length $0$. Inductively, we assume that  $(D_{\widehat{Z}, \widehat{H}})^2(\nu) =0$ whenever $\nu$ is such a monomial of length at most $r-1$.

Now let $\nu = \ :(\partial^{i_1} \Gamma^A) \cdots (\partial^{i_r} \Gamma^A) \eta:$ be a monomial of length $r$ as above, and write 
$$\nu =  \ :(\partial^{i_1} \Gamma^A) \nu',\qquad \nu' = \ :(\partial^{i_2} \Gamma^A) \cdots (\partial^{i_r} \Gamma^A) \eta:.$$ By the same calculation as \eqref{irtfirst}, we have
\begin{equation} \begin{split} \label{a0first}(D_{\widehat{Z}, \widehat{H}})^2 (\nu)   = & (D_{\widehat{Z}, \widehat{H}})^2 (:(\partial^{i_1} \Gamma^A) \nu':) 
\\  = &  \ :(D_{\text{Der}})^2 (\partial^{i_1}  \Gamma^A) \nu':
\\ & \ +  (D_{\text{NDer}}  D_{\widehat{H}} + D_{\widehat{H}}  D_{\text{NDer}})(:(\partial^{i_1}  \Gamma^A) \nu':))
\\ & \ - :(\partial^{i_1}  \Gamma^A) (D_{\text{NDer}}  D_{\widehat{H}} + D_{\widehat{H}}  D_{\text{NDer}})(\nu'): .\end{split} \end{equation}
We compute
\begin{equation} \label{a0second} \begin{split}   \ :((D_{\text{Der}})^2 (\partial^{i_1}  \Gamma^A)) \nu':\  & =   \  - :(\partial^{i_1+1} \widehat{H}^2) \nu':, 
\\  (D_{\text{NDer}}  D_{\widehat{H}} + D_{\widehat{H}}  D_{\text{NDer}})(:(\partial^{i_1}  \Gamma^A) \nu':)) & =  \   :(\partial^{i_1+1} \widehat{H}^2) \nu': +  :(\partial^{i_1}  \Gamma^A) (D_{\text{NDer}}  D_{\widehat{H}} + D_{\widehat{H}}  D_{\text{NDer}})(\nu'):. \end{split}
\end{equation} The claim is immediate from \eqref{a0first} and \eqref{a0second}. \end{proof}

We are now ready to prove Theorem \ref{sec5main}.

\noindent {\it Proof of Theorem \ref{sec5main}}. We have shown this for the sector $\mathcal{A}^{\text{ch}, \widehat{H}}_0(\widehat{Z})$, so it suffices now to prove it for $\mathcal{A}^{\text{ch}, \widehat{H}}_n(\widehat{Z})$ for all $n \neq 0$. First, we check this on the element $s^n_{\alpha}$ expressed in local coordinates. Since $A = A_{\alpha,\text{bas}} + d \theta_{\alpha}$, we have 
\begin{equation} \label{snfirst} dA = \widehat{H}^2 = d A_{\alpha,\text{bas}}.\end{equation} 
Using \eqref{ope:exotic}, we compute
\begin{equation} \label{snsecond} D_{\text{NDer}}(:\widehat{A} \widehat{H}^2 s^n_{\alpha}:) = -  n :\widehat{H}^2 s^n_{\alpha}:. \end{equation} Combining \eqref{snfirst} and \eqref{snsecond}, we obtain
\begin{equation} \begin{split}  (D_{\text{Der}})^2 (s^n_{\alpha}) = & \ n:\widehat{H^2} s^n_{\alpha}:, \\ (D_{\text{NDer}}  D_{\widehat{H}} + D_{\widehat{H}}  D_{\text{NDer}})(s^n_{\alpha}) & = -n :\widehat{H^2} s^n_{\alpha}:. \end{split} \end{equation}
 It is immediate that $(D_{\widehat{Z}, \widehat{H}})^2(s^n_{\alpha}) = 0$.

Next, a general element $\nu = \mathcal{A}^{\text{ch}, \widehat{H}}_n(\widehat{Z})$ has the form $\nu = \ :s^n_{\alpha} \eta:$ for some $\eta \in \mathcal{A}^{\text{ch}, \widehat{H}}_0(\widehat{Z})$. Since $D_{\widehat{Z}, \widehat{H}}(\eta) = 0$, the same argument as previous lemma shows that $D_{\widehat{Z}, \widehat{H}}(:s^n_{\alpha} \eta:) = 0$. This completes the proof of Theorem \ref{sec5main}.

\section{The case of trivial bundles}
In this section, we assume that both circle bundles $Z$ and $\widehat{Z}$ are trivial, 
$$Z=M\times \TT, \qquad \widehat Z=M\times \widehat{\TT},$$ and that both fluxes $H, \widehat H$ are zero. Then $Z, \widehat{Z}$ have global coordinates $\theta, \widehat{\theta}$ in the circle directions which are defined up to shifts by $2\pi i k$ for $k \in \mathbb{Z}$. The connection forms $A, \widehat{A}$ can be identified with $d\theta, d\widehat{\theta}$, respectively. 

Let $\omega_{-n}\in \Omega^{\bar{0}}(Z)_{-n}$ be an element of even degree. It has the form $ \omega_{-n} = (\lambda_0+\lambda_1d\theta)e^{-n\theta}$, where $\lambda_0, \lambda_1$ are forms on $M$. Then since $\lambda_0$ is even and $\lambda_1$ is odd, by definition we have
$$\tau_n (\omega_{-n})= (\lambda_0 d\widehat\theta+\lambda_1)s^{n}, \qquad \widehat \sigma_n((\lambda_0 d\widehat\theta+\lambda_1)s^{n})=- (\lambda_0+\lambda_1d\theta)e^{ -n\theta}.$$
Suppose $d\omega_{-n}=0$. We then have 
$$d\lambda_0=0, \qquad d\lambda_1 - n\lambda_0=0.$$
Then
$$(d-\iota_{n\widehat v} )\tau_n (\omega_{-n})= (d \lambda_0)(d \widehat{\theta}) s^n + (d \lambda_1) s^n - n \lambda_0 s^n = 0,$$ 
i.e. $\tau_n (\omega_{-n})$ is exotic equivariant closed (in this case equivariant closed).

If $n\neq 0$, one shows that 
$$d\left(\frac{1}{n}\lambda_1e^{- n\theta} \right)=(\lambda_0+\lambda_1d\theta)e^{- n\theta}=\omega_{-n},$$
i.e. $\omega_{-n}$ is $(d+H)$-exact (in this case $d$-exact).
The odd case is similar and is omitted.

We now consider the chiral setting. Since $\widehat{H}$ vanishes, the formula for the differential $D_{\widehat{Z}, \widehat{H}} = D_{\widehat{Z}, \widehat{0}}$ simplifies as follows.
$$D_{\widehat{Z}, \widehat{0}}  =  D_{\widehat{Z}} +  D^2 + D^4, \qquad D^2 =    -  (:\iota_{\widehat{A}} L_A:)_{(0)},\qquad  D^4 =  (:\iota_{\widehat{A}} L_A:)_{(1)}.$$
Unlike the general case, note that $\cA^{\text{ch},\widehat{H}}(\widehat{Z})$ is in fact graded by conformal weight, not just filtered, and $D_{\widehat{Z}, \widehat{0}}$ preserves the weight grading.

Recall that $\Omega^{\text{ch},H}(Z) = \Omega^{\text{ch},0}(Z)$ admits a contracting homotopy for the differential $D$; there is a field $G$ whose mode $G_0$ is globally defined, and $[D,G_0] = L_0$, where $L_0$ denotes the conformal weight grading operator. This shows that the cohomology vanishes in positive weight. Since $Z = M \times \TT$, we can write $G_0$ as the sum of two commuting operators $$G_0 = G^M_0 + :\iota_{d\theta} \partial \theta:.$$ Note that even though the coordinate function $\theta$ is only defined up to integer shifts, both the contraction operator $\iota_{d\theta}$ and the derivative $\partial \theta$ are globally defined. Also, $\partial \theta$ can be identified with the element $\Gamma^A$ defined earlier.

Under $\tau^{\text{ch}}$, we have $\tau^{\text{ch}}(G_0) = G^M_0 + (:\widehat{A} \partial \theta:)_{1}$; note that the second term lowers weight by one. We can correct this by adding the operator $ -(:\widehat{A} \partial \theta:)_{1} +  (:\widehat{A} \partial \theta:)_{0}$, which commutes with $\tau^{\text{ch}}(G_0)$. Setting
$$\widehat{G}_0 =  \tau^{\text{ch}}(G_0)  -(:\widehat{A} \partial \theta:)_{1} +  (:\widehat{A} \partial \theta:)_{0},$$ this is easily seen to be a contracting homotopy for $D_{\widehat{Z}, \widehat{0}}$ in the sense that
$$[D_{\widehat{Z}, \widehat{0}}, \widehat{G}_0] = L_0,$$ where $L_0$ is the conformal weight grading operator. It follows that in case of trivial bundles and fluxes, the positive weight cohomology of the exotic complex $(\cA^{\text{ch}, \widehat{0}}(\widehat{Z}), D_{\widehat{Z}, \widehat{0}})$ vanishes. Therefore $\tau^{\text{ch}}$ induces an isomorphism in cohomology in this case even though the intertwining property \eqref{intertwiningofd} still fails in positive weight.\\

\noindent {\bf Acknowledgements}\\
Andrew Linshaw was partially supported by Simons Foundation Grant 635650 and National Science Foundation Grant DMS 2001484. Varghese Mathai was partially supported by funding from the Australian Research Council, through the Australian Laureate Fellowship FL170100020.
The authors are grateful to Maxim Zabzine (Uppsala) for suggesting the problem solved in the paper.


\begin{thebibliography}{99}

\bibitem{AABL}
E.~Alvarez, L.~Alvarez-Gaum\'e, J.L.F.~Barb\'on and Y.~Lozano,
Some global aspects of duality in string theory,
Nucl.\ Phys.\ {\bf B415} (1994) 71-100, [arXiv:hep-th/9309039].



\bibitem{Borcherds}
R. Borcherds. Vertex operator algebras, Kac-Moody algebras and the monster, 
Proc. Nat. Acad. Sci. USA {\bf 83} (1986) 3068-3071.


\bibitem{Borisov}
L. Borisov. Vertex algebras and mirror symmetry. Comm. Math. Phys. 215 (2001), no. 3, 517-557. 
MR1810943  


\bibitem{BEM}
P.~Bouwknegt, J.~Evslin and V.~Mathai, 
T-duality: Topology Change from H-flux,
Commun.\ Math.\ Phys.\ {\bf 249} (2004) 383-415, {[{\tt arXiv:hep-th/0306062}]}.

\bibitem{BEM2}
\bysame, 
On the Topology and Flux of T-Dual Manifolds,
{Phys. Rev. Lett.  \textbf{92}  (2004)   181601},
{[{\tt arXiv:hep-th/0312052}]}.


\bibitem{BHM}
P.~Bouwknegt, K.~Hannabuss and V.~Mathai, 
Nonassociative tori and applications to T-duality. Comm. Math. Phys. 264 (2006), no. 1, 41--69.
{[{\tt arXiv:hep-th/0412092}]}.

\bibitem{Brylinski}	
Jean-Luc Brylinski,
{\it Loop spaces, characteristic classes and geometric quantization},
Progress in Mathematics, {\bf 107}, Birkhauser Boston, Inc., Boston, MA, 1993.
MR1197353

\bibitem{cavalcanti}
G. Cavalcanti and M. Gualtieri. 
Generalized complex geometry and T-duality.
A Celebration of the Mathematical Legacy of Raoul Bott (CRM Proceedings \& Lecture Notes), AMS,
2010, 341--365.

\bibitem{BS}
U.\ Bunke and T.\ Schick,
{ On the topology of T-duality},
{Rev.\ Math.\ Phys.\  \textbf{17}  (2005) 77-112},
{[{\tt arXiv:math.GT.0405132}]}.

\bibitem{Bus}
T.~Buscher,
A symmetry of the string background field equations,
Phys.\ Lett.\ {\bf B194} (1987) 59-62;
Path integral derivation of quantum duality in nonlinear sigma models, 
{\it ibid.}, {\bf B201} (1988) 466-472.

\bibitem{Crilly}
J. Crilly and V. Mathai,
Exotic Courant algebroids and T-duality, J. Geom. Phys. 163 (2021) 104155.
 [arxiv:1909.07127] 


\bibitem{DSK} A. De Sole and V. Kac, Freely generated vertex algebras and non-linear Lie conformal algebras, Comm. Math. Phys. 254 (2005), no. 3, 659-694. 


\bibitem{FBZ} E. Frenkel and D. Ben-Zvi, Vertex Algebras and Algebraic Curves, Math. Surveys and Monographs, Vol. 88, American Math. Soc., 2001.


\bibitem{FLM}
I.B. Frenkel, J. Lepowsky, and A. Meurman, Vertex Operator Algebras and the Monster, Academic Press, New York, 1988.

\bibitem{FMS}
D. Friedan, E. Martinec, and S. Shenker, Conformal invariance, supersymmetry and string theory. Nucl. Phys. B271 (1986) 93-165.


\bibitem{Hori}
K.~Hori, S.~Katz, A.~Klemm, R.~Pandharipande, R.~Thomas, C.~Vafa, R.~Vakil
and E.~Zaslow, { Mirror symmetry},
Clay Mathematics Monographs, 1. American Mathematical Society, Providence, RI
and Clay Mathematics Institute, Cambridge, MA, 2003.




\bibitem{HM15}
F.~Han and V.~Mathai, Exotic twisted equivariant cohomology of loop spaces, twisted Bismut-Chern character and T-duality, 
Comm. Math. Phys., {\bf 337}, no. 1, (2015) 127--150. MR3324158 [{\tt arXiv:1405.1320}].

\bibitem{HM18}
F.~Han and V.~Mathai. 
T-Duality in an H-flux: Exchange of momentum and winding,
Comm. Math. Phys.,  
{\bf 363}, no. 1 (2018) 333-350.
[{\tt arXiv:1710.07274}]



\bibitem{Hori1}  
K.~Hori, 
D-branes, T-duality, and index theory,
Adv.\ Theor.\ Math.\ Phys.\ {\bf 3} (1999) 281-342, [arXiv:hep-th/9902102].

\bibitem{KO}
A. Kapustin and D. Orlov. 
Vertex algebras, mirror symmetry, and D-branes: the case of complex tori.
Comm. Math. Phys.
{\bf 233}, 2003, 79--136. 

\bibitem{K}
V. Kac.
Vertex algebras for beginners.
University Lecture Series,  {\bf 10}, AMS, 1996. Second edition, 1998.

\bibitem{LiI} 
H. Li. Local systems of vertex operators, vertex superalgebras and modules, J. Pure Appl. Algebra 109 (1996), no. 2, 143-195.

\bibitem{LL} 
B. Lian and A. Linshaw, Chiral equivariant cohomology I, Adv. Math. 209 (2007), no. 1, 99-161.


\bibitem{LLS} 
B. Lian, A. Linshaw, and B. Song, Chiral equivariant cohomology II, Trans. Am. Math. Soc. 360 (2008), 4739-4776.


\bibitem{LZ1}
B. Lian and G. Zuckerman, New Perspectives in the BRST algebraic structure of string theory, Comm. Math Phys. 154
(1993) 613-646.

\bibitem{LZ2}
B. Lian and G. Zuckerman, Commutative quantum operator algebras, J. Pure Appl. Algebra 100 (1995) no. 1-3, 117-139.

\bibitem{L} A. Linshaw, Universal two-parameter ${\mathcal W}_{\infty}$-algebra and vertex algebras of type ${\mathcal W}(2,3,\dots, N)$, Compos. Math. 157, no. 1 (2021), 12-82.

\bibitem{LM15}  
A.~Linshaw and V.~Mathai, 
Twisted chiral de Rham complex, generalized geometry, and T-duality.
Commun.\ Math.\ Phys.\ {\bf 339} no. 2 (2015) 663-697.  [arXiv:1412.0166]


\bibitem{MQ}
V.~Mathai and D.~Quillen,
Superconnections, Thom classes, and equivariant differential forms. Topology {\bf 25} no.1 (1986) 85--110.
MR0836726

\bibitem{MR05}  {V. Mathai} and J. Rosenberg, 
T-duality for torus bundles via noncommutative topology
Comm. Math. Phys.,
 {\bf 253} no.3 (2005) 705-721.
 [arXiv:hep-th/0401168]  

 \bibitem{MR05b}  {V. Mathai} and J. Rosenberg, 
 On mysteriously missing T-duals, H-flux and the T-duality group,
pages 350-358, in "Differential Geometry and Physics", editors Mo-Lin Ge and Weiping Zhang, Nankai Tracts in Mathematics, Volume 10, World Scientific 2006.
 [arXiv:hep-th/0409073]
 
\bibitem{MR06}   { {V. Mathai}} and J. Rosenberg, 
T-duality for torus bundles with H-fluxes via noncommutative topology, II: 
the high-dimensional case and the T-duality group,
Adv.\ Theor.\ Math.\ Phys., 
{\bf 10} no. 1 (2006) 123-158.
[ arXiv:hep-th/0508084]



\bibitem{MSV1} F. Malikov, V. Schechtman, and A. Vaintrob. Chiral de Rham complex. Comm. Math. Phys. 204 (1999), no. 2, 439-473.
  MR1704283

\bibitem{MSV2}  F. Malikov and V. Schechtman. Chiral de Rham complex. II. Differential topology, infinite-dimensional Lie algebras, and applications, 149-188, Amer. Math. Soc. Transl. Ser. 2, 194, Amer. Math. Soc., Providence, RI, 1999. MR1729362  



\bibitem{Witten07}
E.~Witten.  
Two-dimensional models with (0,2) supersymmetry: perturbative aspects. 
Adv. Theor. Math. Phys. 11 (2007), no. 1, 1-63. 
MR2320663




\end{thebibliography}
\end{document}